\documentclass[a4paper,11pt]{amsart} 
\usepackage{amsmath,amssymb,amsthm} 
\usepackage{times}
\usepackage[square,comma,numbers,sort&compress]{natbib}
\theoremstyle{plain}
\newtheorem{theorem}{Theorem}[section]
\newtheorem{proposition}[theorem]{Proposition}
\newtheorem{lemma}[theorem]{Lemma}

\theoremstyle{definition}

\theoremstyle{remark}
\newtheorem{remark}[theorem]{Remark}

\numberwithin{equation}{section}

\newcommand{\p}{\partial}
\newcommand{\ve}{\varepsilon}
\newcommand{\f}{\frac}

\newcommand{\na}{\nabla}
\newcommand{\la}{\lambda}
\newcommand{\al}{\alpha}

\renewcommand{\t}{\tilde}
\renewcommand{\o}{\omega}
\newcommand{\vp}{\varphi}
\renewcommand{\th}{\theta}

\newcommand{\g}{\gamma}

\newcommand{\si}{\sigma}

\newcommand{\dl}{\delta}

\renewcommand{\a}{q_0^{-\f 2{\gamma-1}}}
\newcommand{\ds}{\displaystyle}

\allowdisplaybreaks

\title[Blowup of smooth solutions]{Blowup of smooth solutions for
  general 2-D \\ quasilinear wave equations with small initial data}

\author[B.-B.~Ding]{Bingbing Ding} \address{Department of Mathematics
  and IMS, Nanjing University, Nanjing 210093, P.R.~of China}
\email{balancenjust@yahoo.com.cn}

\author[I.~Witt]{Ingo Witt}
\address{Mathematical Institute, University of G\"ottingen, D-37073
  G\"ottingen, Germany} 
\email{iwitt@uni-math.gwdg.de}

\author[H.-C.~Yin]{Huicheng Yin} \address{Department of Mathematics
  and IMS, Nanjing University, Nanjing 210093, P.R.~of China}
\email{huicheng@nju.edu.cn}

\thanks{Bingbing Ding and Huicheng Yin were supported by the NSFC
  (No.~10931007, No.~11025105) and by the Priority Academic Program
  Development of Jiangsu Higher Education Institutions. Ingo Witt was
  partly supported by the DFG via the Sino-German project ``Analysis
  of PDEs and Applications.''}



\begin{document}


\begin{abstract}
For the 2-D quasilinear wave equation $\ds\sum_{i,j=0}^2g_{ij}(\na
u)\p_{ij}u=0$ with coefficients independent of the solution $u$, a
blowup result for small data solutions has been established in
\cite{1,2} provided that the null condition does not hold and a
generic nondegeneracy condition on the initial data is fulfilled.  In
this paper, we are concerned with the more general 2-D quasilinear
wave equation $\ds\sum_{i,j=0}^2g_{ij}(u, \na u)\p_{ij}u=0$ with
coefficients that depend simultaneously on $u$ and $\na u$. When the
null condition does not hold and a suitable nondegeneracy condition on
the initial data is satisfied, we show that smooth small data
solutions blow up in finite time. Furthermore, we derive an explicit
expression for the lifespan and establish the blowup mechanism.
\end{abstract}


\maketitle


\section{Introduction and main results}

In this paper, we discuss blowup of small data smooth
solutions of 2-D quasilinear wave equations
\begin{equation}\label{1-1}
\left\{ \enspace
\begin{aligned}
& \sum_{i,j=0}^2g_{ij}(u, \na u)\p_{ij}u=0,\\
& (u(0,x), \p_t u(0,x))= \ve (\vp_0(x), \vp_1(x)),
\end{aligned}
\right.
\end{equation}
where $x_0=t$, $x=(x_1, x_2)$, $\na=(\p_0, \p_1, \p_2)$, $\ve>0$ is
small, $\vp_i(x)\in C_0^{\infty}(B(0, M))$ ($i=0,1$) with $B(0, M)$
being the disk of radius $M>0$ centered at the origin, and the
coefficients $g_{ij}(u, \na u)$ ($0\le i, j\le 2$) are $C^{\infty}$
smooth in their arguments.

Without loss of generality, we write
\[
g_{ij}(u, \na u)=c_{ij}+d_{ij}u+
\sum_{k=0}^2e_{ij}^k\p_ku+O(|u|^2+|\na u|^2),
\]
where $c_{ij}=c_{ji}$, $d_{ij}=d_{ji}$, and $e_{ij}^k$ are constants,
$\ds\sum_{i,j=0}^2c_{ij}\p_{ij}=\p_t^2-\Delta$, $d_{00}=0$, and
$d_{ij}\not=0$ for at least one $(i,j)\not=(0,0)$.

In addition, we assume that $\ds\sum_{i,j,k=0}^2e_{ij}^k\p_ku\p_{ij}u$
does not satisfy the null condition. This means that
$\ds\sum_{i,j,k=0}^2e_{ij}^k\xi_i\xi_j\xi_k\not\equiv 0$ for the
variables $(\xi_0, \xi_1, \xi_2)$ with $\xi_0^2=\xi_1^2+\xi_2^2$ and
$(\xi_1, \xi_2)\not=0$ (see \cite{5, 15} for a definition of the null
condition).

We introduce polar coordinates $(r, \th)$ in ${\mathbb R}^2$,
\[
x_1=r\cos\th, \quad x_2=r\sin\th,
\]
where $r=\sqrt{x_1^2+x_2^2}$ and $\th\in[0,2\pi]$. We will need
the function
\begin{equation}\label{1-2}
F_0(\si,\th)=F_0(\si,
\o)\equiv\ds\f{1}{2\pi\sqrt{2}}\int_{\si}^{+\infty}\ds\f{R(s,\o;
\vp_1)-\p_sR(s,\o; \vp_0)}{\sqrt{s-\si}}\,ds,
\end{equation}
where $\si\in {\mathbb R}$, $\o\equiv(\o_1, \o_2)=(\cos\th, \sin\th)$,
and $R(s,\o;v)$ is the Radon transform of the smooth function $v(x)$,
i.e., $R(s,\o;v)=\int_{x\cdot\o=s}v(x)\,dS$.  From \cite[Theorem~6.2.2
  and (6.2.12)]{10}, one has that $F_0(\si,\th)\not\equiv 0$ unless
$(\vp_0(x), \vp_1(x))\equiv 0$. Furthermore, $F_0(\si,\th)\equiv 0$ for
$\si\ge M$ and $\ds\lim_{\si\to -\infty}F_0(\si,\th)=0$.

Set
\[
  F_1(\si,\th)=\biggl(\sum_{i,j=0}^2d_{ij}\hat{\o}_i\hat{\o}_j\biggr)
\p_\si F_0(\si,\th), \;
F_2(\si,\th)=\biggl(\sum_{i,j,k=0}^2e_{ij}^k\hat{\o}_i\hat{\o}_j
\hat{\o}_k\biggr)\p_\si^2F_0(\si,\th),
\]
where $(\hat{\o}_0,\hat{\o}_1,\hat{\o}_2)=(-1,\o_1,\o_2)$. Define the
function
\begin{equation*}
G_0(\si,\th)=
\ds\f{1}{F_1(\si,\th)}\ln
\biggl(1+\ds\f{F_1(\si,\th)}{F_2(\si,\th)}\biggr),
\quad(\si,\th)\in A,
\end{equation*}
where
\[
A=\biggl\{(\si,\th)\in (-\infty, M)\times [0, 2\pi]\colon
F_1(\si,\th)\not =0, F_2(\si,\th)\not =0,
1+\ds\f{F_1(\si,\th)}{F_2(\si,\th)}>0\biggr\}.
\]

Denote
\begin{equation}\label{1-4}
\tau_0=\inf_{(\si,\th)\in B}G_0(\si,\th),
\end{equation}
where $B=\{(\si, \th)\in A \colon G_0(\si,\th)>0\}$. We emphasize that
$0<\tau_0<\infty$ holds provided that $(\vp_0(x), \vp_1(x))\not\equiv
0$ (see Lemma~\ref{lem2-1} below).

We further require the following non-degeneracy condition to hold:
\begin{gather}
\text{There exists a unique  minimum point $(\si_0, \th_0)\in B$
such that} \tag{ND} \\
\text{$\tau_0=G_0(\si_0, \th_0)$ and the Hessian matrix
$(\na_{\si, \th}^2G_0)(\si_0,\th_0)$ is positive definite.}
\notag
\end{gather}
See Remark~\ref{rem2-1} below for an argument that (ND) generically
holds.

\begin{theorem}[Main Theorem]\label{thm1-1}
Let $(\vp_0(x), \vp_1(x))\not\equiv 0$ and assume that\/ \textup{(ND)}
holds. Then problem \eqref{1-1} has a unique $C^\infty$ solution
$u(t,x)$ for $0\leq t <T_\varepsilon$, where $T_{\ve}$ is its
lifespan that, in addition, satisfies
\begin{equation*}
\lim_{\varepsilon\rightarrow 0}\varepsilon\sqrt{T_\varepsilon}
=\tau_0>0.
\end{equation*}
Moreover, there exist a point $M_\ve=(T_\ve, x_\ve)$ and a constant
$C>1$ independent of $\ve$ such that

\textup{(i)} \ $u(t,x)\in C^1([0,T_\ve]\times{\mathbb R}^2)$ and
$\|u\|_{L^{\infty}((0,T_\ve)\times\mathbb
R^2)}+\|\na_{t,x}u\|_{L^{\infty}((0,T_\ve)\times{\mathbb R}^2)}\leq
C\ve$.

\textup{(ii)} \ $u\in C^2(([0,T_\ve]\times{\mathbb
  R}^2)\setminus\{M_{\ve}\})$ and it satisfies, for $0\leq t <T_\ve$,
\begin{equation}\label{1-6}
\ds\f{1}{C(T_\ve-t)}\le\|\nabla_{t, x}^2 u(t,\cdot)\|_{L^\infty(\mathbb
R^2)}\leq\f{C}{T_\ve-t}.
\end{equation}
\end{theorem}

\begin{remark}\label{rem1-1}
For smooth small data solutions of
$\p_t^2u-\ds\sum_{i=1}^2\p_i(c_i^2(u)\p_iu)=0$, it has been shown in
\cite{6} that the blowup mechanism is of ODE type. This means that
$\na_{t,x}u$ develops a singularity at the lifespan time $T_{\ve}$,
while $u(t,x)$ remains continuous up to time $T_{\ve}$. As
Theorem~\ref{thm1-1} illustrates, here the blowup mechanism for smooth
small data solutions of \eqref{1-1} is of geometric type. This means
that only $\na_{t,x}^2u$ develops a singularity at time $T_{\ve}$,
while both $u(t,x)$ and $\na_{t,x}u$ remain continuous up to time
$T_{\ve}$. Theorem~\ref{1-1} is similar in scope to the ``lifespan
theorems'' of \cite{1, 2}, where 2-D nonlinear wave equations
$\p_t^2v-\Delta_xv +\ds\sum_{0\le i,j,k\le 2}g_{ij}^k\p_k
v\,\p_{ij}v=0$, wtih the $g_{ij}^k$ being constants, have been studied
in cases when the null condition does not hold.
\end{remark}

\begin{remark}\label{rem1-2}
For the 3-D wave equation $\p_t^2 u-\ds\sum_{\stackrel{0\le i\le j\le
    3}{(i, j)\neq(0, 0)}}(\dl_{ij}+d_{ij}u)\p_{ij} u=0$, where
$d_{ij}\in{\mathbb R}$ and $d_{ij}\not=0$ for some $(i, j)\not=(0,
0)$, with small initial data $(u(0,x), \p_t u(0,x))$ $=(\ve u_0(x),
\ve u_1(x))$, it has been shown in \cite{4, 19, 20} that smooth
solutions exist globally. On the other hand, for $n$-dimensional
nonlinear wave equations ($n=2,3$) with coefficients depending only on
the gradient of the solution, $\p_t^2 u-c^2(\p_t u)\Delta u=0$ and,
more generally, $\ds\sum_{i,j=0}^ng_{ij}(\na u)\p_{ij}u=0$, where
$t=x_0$, $x=(x_1, ..., x_n)$, $g_{ij}(\na
u)=c_{ij}+\ds\sum_{k=0}^nd_{ij}^k\p_ku+O(|\na u|^2)$, and the linear
part $\ds\sum_{i,j=0}^nc_{ij}\p_{ij}u$ is strictly hyperbolic with
respect to time $t$, it is known that small data smooth solutions
exist globally or almost globally if corresponding null conditions
hold (see \cite{3, 5, 15, 21, 22, 23} and the references therein),
otherwise smooth small data solutions blow up in finite time (see
\cite{1, 2, 6, 7, 8, 9, 10, 11, 12, 13, 14, 18}, and so forth).
\end{remark}

\begin{remark}\label{rem1-3}
From the results of \cite{17} it follows that the lifespan $T_{\ve}$
of smooth small data solutions of \eqref{1-1} satisfies $T_{\ve}\ge
C/\ve$ for small $\ve>0$. Similar to the proof of \cite[Theorem
  2.3]{18}, where 3-D quasilinear wave equations $\square
u=\ds\sum_{i=0}^3\p_iG_i(u, \na u) +G_4(\na u, \na^2u)$ with the $G_i$
$(0\le i\le 4)$ being quadratic forms have been treated, one can
further obtain that $T_{\ve}\ge C/\ve^2$. Here, in
Theorem~\ref{thm1-1}, $\tau_0/\ve^2$ is shown to be the precise bound
for the lifespan $T_\ve$.
\end{remark}

\begin{remark}\label{rem1-4}
For rotationally symmetric solutions $u(t,r)$ of Eq.~ \eqref{1-1}
(i.e., when $\vp_0(x)$, $\vp_1(x)$ and (1.1) are rotationally symmetric) it
can be shown by arguments as in \cite{7} and \cite{12} that
\eqref{1-6} continues to hold even without assumption (ND).
\end{remark}

\begin{remark}\label{rem1-5}
As $d_{00}=0$ and $d_{ij}=d_{ji}\not=0$ for at least one $(i,
j)\not=(0, 0)$, one has that under the restriction that $\xi_0=-1$ and
$\xi_1^2+\xi_2^2=1$, $\ds\sum_{i,j=0}^2d_{ij}\xi_i\xi_j\not=0$ holds
except for finitely many points $(\xi_1^0, \xi_2^0)$. An analogous
statement is true for
$\ds\sum_{i,j,k=0}^2e_{ij}^k\xi_i\xi_j\xi_k$. These simple facts will
be used in the proof of $\tau_0>0$ in Lemma~\ref{lem2-1} below.
\end{remark}

As in \cite{1, 2}, we are able to provide a more accurate
description of the behavior of solutions $u$ near the blowup point
$M_{\ve}$:

\begin{theorem}[Geometric Blowup Theorem]\label{thm1-2}
Choose constants $\tau_1$, $A_0$, $A_1$ and $\dl_0$ such that
$0<\tau_1<\tau_0$, $A_0<\si_0<A_1<M$, $A_0$ and $A_1$ are close to
$\si_0$, and $\dl_0>0$ is sufficiently small. Denote by $\mathcal D$
the domain
\[
\mathcal D\equiv\{(s,\th,\tau)\colon A_0\le s\le A_1, \,
\th_0-\dl_0\le\th\le\th_0+\dl_0, \, \tau_1\le\tau\le\tau_\ve\},
\]
where $\tau_\ve=\ve\sqrt{T_\ve}$. Then there exist a subdomain
$\mathcal D_0$ of $\mathcal D$ containing a point
$m_\ve=(s_\ve,\th_\ve,\tau_\ve)$ and functions $\phi(s,\th,\tau), \,
w(s,\th,\tau), \,v(s,\th,\tau) \in C^3(\mathcal D_0)$ such that, in
the domain $\mathcal D_0$, $\phi$ satisfies
\begin{equation}\label{H}
\left\{ \enspace
\begin{aligned}
\p_s\phi(s,\th,\tau)\geq
0,\quad\p_s\phi(s,\th,\tau)=
0 \ \Longleftrightarrow \ (s,\th,\tau)=m_\ve,\\
\p_{\tau s}\phi(m_\ve)<0,\quad\nabla_{s,\th}\p_s\phi(m_\ve)=0,
\quad\nabla^2_{s,\th}\p_s\phi(m_\ve)>0.
\end{aligned} \tag{H}
\right.
\end{equation}
Moreover, $\p_sw=v\p_s\phi$ and
\begin{equation}\label{1-7}
\p_sv(m_\ve)\neq 0.
\end{equation}
Let $G(\si,\th,\tau)$ be defined by
$G(\Phi)=w(s,\th,\tau)$ and $(\p_{\si}G)(\Phi)=v(s,\th,\tau)$ in the
domain $\Phi(\mathcal D_0)$, where $\Phi$ is the map
$\Phi(s,\th,\tau)=(\phi(s,\th,\tau),\th,\tau)$. Then
\[
  u(t,x)=\ds\f{\ve}{\sqrt{r}}\,G(r-t,\th,\ve\sqrt{t})
\]
solves \eqref{1-1} for $t<T_{\ve}$ near the point
\[
  M_{\ve}=\Phi(m_{\ve})=\left(T_{\ve},
 (T_{\ve}+\phi(s_{\ve}, \th_{\ve}, \tau_{\ve}))\cos\th_{\ve},
 (T_{\ve}+\phi(s_{\ve}, \th_{\ve}, \tau_{\ve}))\sin\th_{\ve}\right).
\]
\end{theorem}

\begin{remark}\label{rem1-6}
As in \cite{1, 2}, Theorem \ref{thm1-2} provides a more accurate
description of the solution $u$ for $t< T_{\ve}$ near the blowup point
$M_{\ve}$ than the one given in Theorem~\ref{thm1-1}. For instance,
$G, \, \na G\in C(\Phi(\mathcal D_0))$ can be directly seen from
$(w(s,\th,\tau), \, v(s,\th,\tau))\in C^3(\mathcal D_0)$ and condition
\eqref{H}. Moreover, it follows from
$\p_{\si}^2G=\ds\f{\p_sv}{\p_s\phi}$,
$u(t,x)=\ds\f{\ve}{\sqrt{r}}\,G(r-t,\th,\ve\sqrt{t})$, condition
\eqref{H}, and a direct verification that there exists a positive
constant $C$ independent of $\ve$ such that
\[
\f{1}{C(T_{\ve}-t)}\le \|\na_{t,x}^2u\|_{L^{\infty}(\Phi(\mathcal D_0))}
\le\f{C}{T_{\ve}-t}.
\]
\end{remark}

Let us briefly comment on the proofs of Theorems~\ref{thm1-1}
and~\ref{thm1-2}. First we establish the lower bound on the lifespan
$T_{\ve}$.  As in \cite[Chapter~6]{10} and \cite{6}, this lower bound
is obtained by constructing an approximate solution $u_a(t,x)$ of
\eqref{1-1} and estimating the difference of the exact solution
$u(t,x)$ and $u_a(t,x)$ by applying the Klainerman-Sobolev inequality
from \cite{16} and further establishing a delicate energy estimate.
Next we show the upper bound on $T_{\ve}$. Motivated by the
``geometric blowup'' method of \cite{1,2} for handling quasilinear
wave equation $\p_t^2u-\Delta_xu+\ds\sum_{0\le i,j,k\le 2}g_{ij}^k\p_k
u\,\p_{ij}u=0$, we introduce the blowup system of \eqref{1-1} to study
the lifespan $T_{\ve}$ and blowup mechanism. That is, by performing a
singular change $\Phi$ of coordinates in the domain $\mathcal
D=\bigl\{(\si, \th, \tau)\colon -C_0\le\si\le M,\, 0\le\th\le 2\pi,
$ \linebreak $0<\tau_1\le\tau\le \tau_{\ve} \bigr\}$,
\[
(s, \th, \tau)\mapsto (\phi(s,\th,\tau),\th, \tau),
\]
where
\[
\text{$\si=\phi(s,\th,\tau_1)$ and $\p_s\phi=0$ at some point,}
\]
while $\si=r-t$, $\tau=\ve\sqrt{t}$, and $C_0>0$ is a fixed constant,
and by setting $G(\Phi)=w(s,\th, \tau)$ and
$(\p_{\si}G)(\Phi)=v(s,\th,\tau)$, we obtain a nonlinear partial
differential system for $(\phi, w, v)$ from the ansatz
$u(t,x)=\ds\f{\ve}{\sqrt{r}}\, G(r-t, \th, \ve \sqrt{t})$ and the
equation in \eqref{1-1}. This blowup system for \eqref{1-1} can be
shown to admit a unique smooth solution $(\phi, w, v)$ for
$\tau\le\tau_{\ve}$, where the pair $(\phi, v)$ satisfies properties
\eqref{H} and \eqref{1-7} of Theorem~\ref{thm1-2}. This enables us to
determine the blowup point at time $t=T_{\ve}$ and give a complete
asymptotic expansion of $T_{\ve}$ as well as a precise description of
the behavior of $u(t,x)$ close to the blowup point. In the process of
treating the resulting blowup system, as in \cite{1,2}, we will use
the Nash-Moser-H\"ormander iteration technique to overcome the
difficulties introduced by the free boundary $t=T_{\ve}$ and the
complicated nonlinear blowup system. To this end, the linearized
blowup system is solved first. Note that due to the simultaneous
appearance of $u$ and $\na u$ in the coefficients $g_{ij}(u, \na u)$,
the resulting blowup system of \eqref{1-1} exhibit some features
different from those in \cite{1, 2} (see \eqref{3-14}--\eqref{3-15}
below). For instance, compared with the linearized blowup system of
$\p_t^2u-\Delta u+\ds\sum_{0\le i, j, k\le 2}g_{ij}^k\p_ku\p_{ij}u=0$
in \cite{2}, certain coefficients $\al_1$ and $\al_2$ in \eqref{3-14}
are not small. Moreover, there are more terms in \eqref{3-15} to be
dealt with than in the corresponding equation ($3.1.1_b$) of \cite{2}.
Thanks to multipliers chosen as in \cite{1,2}, by an integration by
parts we derive energy estimates of the solutions of the linearized
blowup system directly and subsequently show its solvability.  Based
on these estimates and the standard Nash-Moser-H\"ormander iteration
technique, the proof of Theorem~\ref{thm1-2} is accomplished.

The paper is organized as follows: In Section~2, we construct an
approximate solution $u_a(t,x)$ of \eqref{1-1}, as in \cite{10}, and
establish some related estimates. These estimates allow us to obtain
the required lower bound on the lifespan $T_{\ve}$. In Section~3, the
blowup system of \eqref{1-1} is constructed and solved. This allows us
to prove Theorem~\ref{thm1-2}. The proof of Theorem~\ref{thm1-1} is
carried out in Section~4.

\smallskip

\textit{Notation}\/: \ Throughout this paper, we will denote by $Z$
any of the Klainerman vector fields in $\mathbb R_t^+\times{\mathbb
  R}^2$, i.e.,
\[
\p_t, \enspace \p_1, \enspace \p_2, \enspace
S=t\p_t+\ds\sum_{j=1}^2x_j\p_j, \enspace H_i=x_i\p_t+t\p_i,
\enspace i=1, 2, \enspace R=x_1\p_2-x_2\p_1,
\]
where $\p$ stands for $\p_t$ or $\p_i$ ($i=1,2$), and $\na_x$ stands
for $(\p_1, \p_2)$.


\section{The lower bound on the lifespan $T_\varepsilon$}

In this section, we will establish the lower bound of $T_\varepsilon$
as $\ve\to0$ for smooth solutions $u$ of problem \eqref{1-1}. This is
done as in the proof of \cite[Theorem 6.5.3]{10} by constructing an
approximate solution $u_a$ of \eqref{1-1} and then by estimating the
difference $u-u_{a}$. Eventually, one derives the lower bound of
$T_{\ve}$ by a continuous induction argument. The new ingredient in
this procedure is how to construct the approximate solution and to
look for the precise blowup time of the nonlinear profile equation of
\eqref{1-1}, then how to treat both the solution $u$ and its gradient
$\na u$ rather than only the gradient $\na u$ of the solution, as in
\cite{10}. Although some of the arguments are analogous to those in
\cite{6}, for the reader's convenience and in order to obtain the
upper bound of $T_{\ve}$ later, we will provide a complete proof.

Set the slow time variable to $\tau=\ve\sqrt{1+t}$ and assume that the
solution $u$ of \eqref{1-1} is approximated by
\[
\f{\varepsilon}{\sqrt r}\,V(q,\o,\tau),\quad r>0,
\]
where $q=r-t$, $\o=(\o_1, \o_2)=x/r\in\mathbb S^1$, and
$V(q,\o,\tau)$ solves by the equation
\begin{equation}\label{2-1}
\left\{ \enspace
\begin{aligned}
&\partial_{q\tau} V-\ds\sum_{i,j=0}^2\biggl(d_{ij}V+
\ds\sum_{k=0}^2e_{ij}^k\hat{\o}_k\p_q V\biggr)
\hat{\o}_i\hat{\o}_j\p_q^2 V=0,\\
&V(q,\o,0)=F_0(q,\o),\\
&\operatorname{supp}V\subseteq \{q\leq M\},\\
\end{aligned}
\right.
\end{equation}
where $(\hat{\o}_0,\hat{\o}_1, \hat{\o}_2)=(-1,\o_1,\o_2)$, and
$F_0(q, \o)$ has been defined in \eqref{1-2}.

Before studying the blowup problem for \eqref{2-1}, we are required to
establish the following two lemmas:

\begin{lemma}\label{lem2-1}
Let $\tau_0$ be given by \eqref{1-4}. Then $0<\tau_0<\infty$ provided
that $(\vp_0(x), \vp_1(x))\not\equiv 0$.
\end{lemma}

\begin{proof}
We divide the proof into four steps.

\smallskip

\textit{Step 1\/}: $D\equiv\{(\si, \th)\in (-\infty, M)\times [0,
2\pi]\colon F_2(\si, \th)>0\}\not=\emptyset$.

\smallskip

If $D=\emptyset$, then $F_2(\si, \th)\le 0$ for all $(\si, \th)\in
(-\infty, M)\times [0, 2\pi]$. This means that
$\p_{\si}^2F_0(\si,\th)\le 0$ if
$\ds\sum_{i,j,k=0}^2e_{ij}^k\hat{\o}_i\hat{\o}_j \hat{\o}_k\ge 0$ or
$\p_{\si}^2F_0(\si,\th)\ge 0$ if
$\ds\sum_{i,j,k=0}^2e_{ij}^k\hat{\o}_i\hat{\o}_j \hat{\o}_k\le
0$. (Note that $\ds\sum_{i,j,k=0}^2e_{ij}^k\hat{\o}_i\hat{\o}_j
\hat{\o}_k=0$ holds only for finitely many values of $\th$, since the
bilinear form $\ds\sum_{i,j,k=0}^2e_{ij}^k\p_ku\p_{ij}u$ does not
satisfy the null condition.) Without loss of generality, one can
assume that $\p_{\si}^2F_0(\si,\th)\le 0$. Then $F_0(\si,\th)\equiv 0$
follows from $\ds\lim_{\si\to -\infty}F_0(\si, \th)=0, \ds\lim_{\si\to
  -\infty}\p_{\si}F_0(\si, \th)=0$, and $F_0(\si,\th)\equiv 0$ for
$\si\ge M$. Thus, $(\vp_0(x), \vp_1(x))\equiv 0$ which is a
contradiction.

\smallskip

\textit{Step 2\/}: $D_1\equiv\{(\si, \th)\in (-\infty, M)\times [0,
  2\pi]\colon F_1(\si, \th)\not=0, F_2(\si, \th)\not=0\}\not=\emptyset$.

\smallskip

Set $D_{11}\equiv\{(\si, \th)\in (-\infty, M)\times [0, 2\pi]\colon
F_1(\si, \th)\not=0\}$ and $D_{12}\equiv\{(\si, \th)\in (-\infty,
M)\times [0, 2\pi]\colon F_2(\si, \th)\not=0\}$. Then $D_1=D_{11}\cap
D_{12}$. If $D_1=\emptyset$, then for each fixed $\th\in [0, 2\pi]$,
the level set $D_1^{\th}=\{\si\in (-\infty, M)\colon (\si, \th)\in
D_1\}$ is empty. Note that $D_1^{\th}=D_{11}^{\th}\cap D_{12}^{\th}$,
where $D_{1i}^{\th}$ ($i=1,2$) is an open set. By Remark~\ref{rem1-5},
we can assume that
$D_{1i}^{\th}=\bigcup_{l=1}^{N_i^{\th}}(a_{i,l}^{\th}, b_{i,l}^{\th})$
($N_i^{\th}$ is finite or infinite), where the intervals
$(a_{i,l}^{\th}, b_{i,l}^{\th})$ $(1\le l\le N_i^{\th})$ are disjoint
for different $l$, moreover, $F_i(a_{il}^{\th}, \th)=F_i(b_{il}^{\th},
\th)=0$ for $i=1,2$. This immediately yields
$\p_{\si}F_0(a_{1l}^{\th}, \th)=\p_{\si}F_0(b_{1l}^{\th}, \th)=0$ and
$\p_{\si}^2F_0(a_{2l}^{\th}, \th)=\p_{\si}^2F_0(b_{2l}^{\th}, \th)=0$
from the expressions for $F_1(\si,\th)$ and $F_2(\si,\th)$,
respectively. Due to $D_1^{\th}=\emptyset$, one has
$\p_{\si}^2F_0(\si,\th)\equiv 0$ and further $\p_{\si}F_0(\si,
\th)\equiv 0$ for $\si\in (a_{1l}^{\th}, b_{1l}^{\th})$. Note that
$\p_{\si}F_0(\si, \th)=0$ on $(-\infty, M)\setminus D_{11}^{\th}$,
hence $\p_{\si}F_0(\si, \th)\equiv 0$ holds for all $\si\in (-\infty,
M)$ which yields $F_0(\si, \th)\equiv 0$ for each $\th\in [0,
  2\pi]$. This, however, contradicts $\vp_0(x)\not\equiv 0$ or
$\vp_1(x)\not\equiv 0$.

\smallskip

\textit{Step 3\/}: \ $B\not=\emptyset$.

\smallskip

It is readily seen that $B=\{(\si, \th)\in (-\infty, M)\times [0,
  2\pi]\colon F_1(\si, \th)\not=0, F_2(\si, \th)>0, F_1(\si,
\th)+F_2(\si, \th)>0\}$. Set $B_1\equiv \{(\si, \th)\in (-\infty,
M)\times [0, 2\pi]\colon F_1(\si, \th)\not=0,$ $F_2(\si, \th)>0\}$ and
denote by $B_1^{\th}$ the level set of $B_1$ for fixed $\th$. Further
write $B_1^{\th}=\bigcup_{l=1}^{N^{\th}}(c_l^{\th}, d_l^{\th})$ if
$B_1^{\th}\neq\emptyset$, where different intervals $(c_l^{\th},
d_l^{\th})$ are disjoint, and $F_1(\si,\th)\not=0$, $F_2(\si,\th)>0$
for $\si\in (c_l^{\th}, d_l^{\th})$. There are four possible cases for
the values of $F_i(\si, \th)$ $(i=1,2)$ at the endpoints of
$(c_l^{\th}, d_l^{\th})$.

\textit{Case (i)\/}: $F_1(c_l^{\th}, \th)=F_1(d_l^{\th}, \th)=0$.

\textit{Case (ii)\/}: $F_1(c_l^{\th}, \th)=F_2(d_l^{\th}, \th)=0$.

\textit{Case (iii)\/}: $F_2(c_l^{\th}, \th)=F_1(d_l^{\th}, \th)=0$.

\textit{Case (iv)\/}: $F_2(c_l^{\th}, \th)=F_2(d_l^{\th}, \th)=0$.

We will prove by contradiction that $B\neq\emptyset$. If
$B=\emptyset$, then $B^{\th}=\emptyset$ for any level set which means
$F_1(\si, \th)+F_2(\si, \th)\le 0$ for $\si\in (c_l^{\th},
d_l^{\th})$. We can assume that $g_1(\th)\equiv
\ds\sum_{i,j=0}^2d_{ij}\hat{\o}_i\hat{\o}_j>0$ and
$g_2(\th)\equiv\ds\sum_{i,j,k=0}^2e_{ij}^k\hat{\o}_i\hat{\o}_j
\hat{\o}_ka>0$ for a fixed $\th$ (Other cases are treated
analogously). One then has:

In case (i),
$\p_{\si}F_0(c_l^{\th},\th)=\p_{\si}F_0(d_l^{\th},\th)=0$. It follows
from $g_1(\th)\p_{\si}F_0(\si, \th)+g_2(\th)\p_{\si}^2F_0(\si,\th)\le
0$ that $\p_{\si}\biggl(e^{\f{g_1(\th)}{g_2(\th)}\si}
\p_{\si}F_0(\si,\th)\biggr) \le 0$ and further $\p_{\si}F_0(\si,
\th)\equiv 0$, which is a contradiction to $\p_{\si}^2F_0(\si,\th)$
$>0$ for $\si\in (c_l^{\th}, d_l^{\th})$.

In case (ii),
$\p_{\si}F_0(c_l^{\th},\th)=\p_{\si}^2F_0(d_l^{\th},\th)=0$. Together
with $\p_{\si}^2F_0(\si,\th)>0$ for $\si\in (c_l^{\th}, d_l^{\th})$,
this yields $\p_{\si}F_0(\si,\th)>0$ and further
$F_1(\si,\th)+F_2(\si,\th)>0$ for $\si\in (c_l^{\th}, d_l^{\th})$
which is a contradiction to the assumption
$F_1(\si,\th)+F_2(\si,\th)\le 0$ for $\si\in (c_l^{\th}, d_l^{\th})$.

In case (iii), $\p_{\si}^2F_0(c_l^{\th},\th)
=\p_{\si}F_0(d_l^{\th},\th)=0$. From $B^{\th}=\emptyset$ and
$\p_{\si}\biggl(e^{\f{g_1(\th)}{g_2(\th)}\si}
\p_{\si}F_0(\si,\th)\biggr) \le 0$ for $\si\in (c_l^{\th},
d_l^{\th})$, one has $\p_{\si}F_0(\si,\th)\ge 0$ and
$F_1(\si,\th)+F_2(\si,\th)>0$ for $\si\in (c_l^{\th}, d_l^{\th})$
which contradicts $F_1(\si,\th)+F_2(\si,\th)\le 0$ for $\si\in
(c_l^{\th}, d_l^{\th})$.

In case (iv), $\p_{\si}^2F_0(c_l^{\th},\th)=
\p_{\si}^2F_0(d_l^{\th},\th)=0$. In view of
$\p_{\si}F_0(\si,\th)\not=0$ and $\p_{\si}^2F_0(\si, \th)>0$ for
$\si\in (c_l^{\th}, d_l^{\th})$, one then obtains
$\p_{\si}F_0(\si,\th)<0$ for $\si\in (c_l^{\th}, d_l^{\th})$ by
$F_1(\si,\th)+F_2(\si,\th)\le 0$. (If $\p_{\si}F_0(c_l^{\th},\th)=0$
or $\p_{\si}F_0(d_l^{\th},\th)=0$, then the proof has been completed
in cases (ii) and (iii), respectively.)  Therefore, there exists an
interval $(e_{1l}^{\th}, e_{2l}^{\th})$ such that $(c_l^{\th},
d_l^{\th})\subset (e_{1l}^{\th}, e_{2l}^{\th})$,
$\p_{\si}F_0(\si,\th)<0$ for $\si\in (e_{1l}^{\th}, e_{2l}^{\th})$,
and $\p_{\si}F_0(e_{1l}^{\th},\th)=\p_{\si}F_0(e_{2l}^{\th},\th)=0$,
moreover, $F_1(\si,\th)+F_2(\si,\th)\le 0$. (Otherwise, if
$F_1(\si,\th)+F_2(\si,\th)>0$ at some point $(\bar\si, \bar\th)$, then
$\p_{\si}^2F_0(\bar\si, \bar\o)>0$ which contradicts
$B^{\th}=\emptyset$).  In this case, by
$\p_{\si}\biggl(e^{\f{g_1(\th)}{g_2(\th)}\si}\p_{\si}F_0(\si,\th)\biggr)
\le 0$ for $\si\in (e_{1l}^{\th}, e_{2l}^{\th})$ one has
$\p_{\si}F_0(\si,\th)\ge 0$ for $\si\in (d_l^{\th}, e_{2l}^{\th})$.
This obviously contradicts $\p_{\si}F_0(\si,\th)<0$ for $\si\in
(e_{1l}^{\th}, e_{2l}^{\th})$.

Collecting the analysis above, one arrives at
$B\neq\emptyset$.

\smallskip

\textit{Step 4\/}: $\tau_0>0$ is a finite number.

\smallskip

Since $B\not=\emptyset$, there exists at least one point
$(\t\si_0, \t\th)\in B$ such $\tau_0\leq G_0(\t\si, \t\th)<\infty$.

Next we show $\tau_0>0$.

Set $z=\ds\f{\p_{\si}F_0(\si, \th)}{\p_{\si}^2F_0(\si, \th)}$ for
$(\si, \th)\in B$. Then $G_0(\si,
\th)=\ds\f{1}{\p_{\si}^2F_0(\si,\th)}\ds\f{\ln(1+z)}{z}$ with
$(\si,\th)\in B$, $z>-1$, and $z\neq0$. If $z\in (-1, 0)$, then
 $G_0(\si,\th)\ge\ds\f{1}{\ds\max_{(\si,\th)\in B}
\p_{\si}^2F_0(\si,\th)}$. If $z\in (0, N]$ for a fixed $N>0$, then
$G_0(\si,\th)\ge\ds\f{C_N}{\ds\max_{(\si,\th)\in
B} \p_{\si}^2F_0(\si,\th)}$, where $0<C_N<1$. If $z>N$ and $N$ is
large enough, then this implies that
$\p_{\si}^2F_0(\si,\th)>0$ is small and
$\p_{\si}F_0(\si,\th)>0$ holds. In this case, $G_0(\si,\th)
=\ds\f{\ln(1+z)}{\p_{\si}F_0(\si,\th)}\ge\f{
\ln(1+N)}{\max(\p_{\si}F_0)}$. Therefore, $\tau_0$ is a positive
constant. 
\end{proof}

\begin{lemma}\label{lem2-2}
Define the function
\[
\t G_0(\si,\th)=\begin{cases}
\ G_0(\si,\th) & \textup{for $(\si,\th)\in A$,}\\
\ \ds\f{1}{F_2(\si,\th)} & \textup{for $(\si,\th)\in D$,}
\end{cases}
\]
where $D=\{(\si,\th)\in (-\infty, M)\times [0, 2\pi]\colon
F_1(\si,\th)=0, F_2(\si,\th)>0\}\neq\emptyset$. Let
\[
\t \tau_0=
\min\biggl\{\tau_0,\inf_{(\si,\th)\in D}\t G_0(\si,\th)\biggr\}.
\]
Then
\[
\t\tau_0=\tau_0.
\]
\end{lemma}

\begin{proof}
It is enough to show $\tau_1\equiv \ds\min_{(\si,\th)\in D}\t
G_0(\si,\th)\ge \tau_0$. Indeed, by the definition of $\tau_1$, there
exists a sequence $\{(\si_n, \th_n)\}_{n\in\mathbb N}\subset D$ such that
$\tau_1>\t G_0(\si_n, \th_n)-1/n$. In addition, it follows
from the definitions of domains $B$ and $D$ that, for each fixed $n$,
there exists a sequence $\{(\si_n^k, \th_n^k)\}_{k\in\mathbb N}\subset B$
such that $\ds\lim_{k\to\infty}(\si_n^k, \th_n^k)=(\si_n,
\th_n)$. Therefore, $\tau_1>\ds\lim_{k\to\infty}G_0(\si_n^k,
\th_n^k)-\f{1}{n}\ge \tau_0-\f{1}{n}$ for any $n\in\mathbb N$ which
yields $\tau_1\ge\tau_0$. 
\end{proof}

\begin{remark}\label{rem2-1}
From the proofs of Lemmas \ref{lem2-1} and \ref{lem2-2}, one sees
that, generically, $\tau_0$~is attained at an interior point of the
set $B$. Indeed, note that $\p B=\p B_1\cup\p B_2\cup\p B_3\cup\p
B_4$, where $\p B_1=\{(\si, \th)\in (-\infty, M]\times [0, 2\pi]\colon
F_1\not=0,\, F_2=0,\, F_1+F_2>0\}$, $\p B_2=\{(\si, \th)\in (-\infty,
M]\times[0, 2\pi]\colon F_1\not=0, F_2>0, F_1+F_2=0\}$, $\p
B_3=\{(\si, \th)\in (-\infty, M]\times [0, 2\pi]\colon F_1=0, F_2>0,
F_1+F_2>0\}$, and $\p B_4=\{(\si, \th)\in (-\infty, M]\times[0,2\pi]\colon
\text{two or three of the values $F_1$, $F_2$, and $F_1+F_2$ at $(\si,
\th)$}$ $\text{are zero}\}$. Generi\-cal\-ly, $\p B_4$ consists of
finitely many points only. Near $\p B_1$ and $\p B_2$, $G_0(\si,\th)$
will be much larger than $\tau_0$, near $\p B_3$, one also has that
$\ds\inf_{\substack{(\si,\th)\in B\\(\si, \th)\to\p B_3}} G_0(\si, \th)\ge
\tau_0$, and only in nongeneric cases $G_0(\si, \th)$ attains its
minimum at the part $\p B_3$ of the boundary.
\end{remark}

Based on Lemmas~\ref{lem2-1}-\ref{lem2-2}, we now state for problem
\eqref{2-1}:

\begin{lemma}\label{lem2-3}
Problem \eqref{2-1} admits a $C^\infty$ solution $V$ for
$0\leq\tau<\tau_0$, where $\tau_0=\ds\min_{(\si,\th)\in
  B}G_0(\si,\th)$.
\end{lemma}

\begin{proof}
Set $w(q,\o,\tau)=\p_q V(q,\o,\tau)$. Then it follows from \eqref{2-1} that
\begin{equation}\label{2-2}
\left\{ \enspace
\begin{aligned}
&\p_\tau w-\ds\sum_{i,j=0}^2\biggl(d_{ij}V+\ds\sum_{k=0}^2e_{ij}^k\hat{\o}_kw\biggr)
\hat{\o}_i\hat{\o}_j\p_q w=0,\quad (q,\tau)\in (-\infty,M]\times[0,\tau_0),\\
&w(q,\o,0)=\p_qF_0(q,\o).
\end{aligned}
\right.
\end{equation}
The characteristics $q=q(\o,\tau; s)$ of \eqref{2-2}
emanating from the point $(\o, 0)$ is defined by
\begin{equation}\label{2-3}
\left\{ \enspace
\begin{aligned}
& \ds\f {dq}{d\tau}(\o,\tau; s)=-\ds\sum_{i,j=0}^2\biggl(d_{ij}V
+\ds\sum_{k=0}^2e_{ij}^k\hat{\o}_kw\biggr)
\hat{\o}_i\hat{\o}_j(q(\o,\tau; s),\o,\tau),\\
& q(\o,0;s)=s.\\
\end{aligned}
\right.
\end{equation}
Along this characteristic curve, one has
\[
\left\{ \enspace
\begin{aligned}
& \ds\frac{dw}{d\tau}(q(\o,\tau; s),\o,\tau)=0,\\
& w(q(\o,0;s),0)=\p_qF_0(s,\o),\\
\end{aligned}
\right.
\]
which yields for $\tau<\tau_0$
\begin{equation}\label{2-4}
w(q(\o,\tau;s),\o,\tau)=\p_qF_0(s,\o)=\p_q V(q(\o,\tau;s),\o,\tau).
\end{equation}

On the other hand, by \eqref{2-3} and \eqref{2-4}, one obtains
\[
\left\{ \enspace
\begin{aligned}
&\p_{\tau
s}^2q(\o,\tau;s)=-\ds\sum_{i,j=0}^2\biggl(d_{ij}\p_qF_0(s,\o)\p_sq(\o,\tau;
s) +\ds\sum_{k=0}^2e_{ij}^k\hat{\o}_k\p^2_qF_0(s,\o)\biggr)
\hat{\o}_i\hat{\o}_ja,\\
&\partial_s q(\o,0;s)=1.
\end{aligned}
\right.
\]
This yields
\[
\p_sq(\o,\tau;s)=\exp\bigl(-F_1(s,\o)\tau\bigr)
\biggl(1+\ds\f{F_2(s,\o)}{F_1(s,\o)}\biggr)
-\f{F_2(s,\o)}{F_1(s,\o)}>0
\]
if $F_1(s,\o)\not=0$ and $\p_sq(\o,\tau,s)$ $=1-\tau F_2(s,\o)>0$ if
$F_1(s,\o)=0$ when $0\leq \tau<\tau_0$, respectively. Then
\[
q(\o,\tau;s) 
=q(\o,\tau;
M)+\int_M^s\biggl(\exp\bigl(-F_1(\rho,\o)\tau\bigr)
\biggl(1+\f{F_2(\rho,\o)}{F_1(\rho,\o)}\biggr)
-\f{F_2(\rho,\o)}{F_1(\rho,\o)}\biggr)d\rho,
\]
where $\ds\lim_{z\to 0}\biggl(e^{z\tau}\left(1-\f{y}{z}
\right)+\f{y}{z}\biggr)=1-\tau y$ has been used.

Note that $q(\o,\tau; M)=M$ such that $V(q,\o,\tau)$ satisfies the
boundary condition $V|_{q=M}=0$. Hence
\begin{align*}
V(q(\o,\tau;s),\o,\tau) &=-\ds\f{\p_\tau q(\o,\tau;
s)}{\ds\sum_{i,j=0}^2d_{ij}\hat{\o}_i\hat{\o}_j}
-\f{\ds\sum_{i,j,k=0}^2e_{ij}^k\hat{\o}_i\hat{\o}_j\hat{\o}_k}
{\ds\sum_{i,j=0}^2d_{ij}\hat{\o}_i\hat{\o}_ja}\,w \\
&=\int_M^s\biggl(\exp(-\tau F_1(\rho,\o))\p_qF_0(\rho,\o) \\
&\qquad\qquad +\bigl(\exp(-\tau
F_1(\rho,\o))-1\bigr)\ds\f{F_2(\rho,\o)\p_qF_0(\rho,\o)}{F_1(\rho,\o)}
\biggr)d\rho.
\end{align*}

By Lemmas \ref{lem2-1}-\ref{lem2-2} and the implicit function theorem,
$s=s(q,\o,\tau)$ is a smooth function for $\tau<\tau_0$. Therefore,
\begin{multline*}
V(q,\o,\tau)=\ds\int_M^{s(q,\o,\tau)}\biggl(\exp(-\tau
F_1(\rho,\o))\p_qF_0(\rho,\o) \\
+\bigl(\exp(-\tau
F_1(\rho,\o))-1\bigr)\ds\f{F_2(\rho,\o)\p_qF_0(\rho,\o)}{F_1(\rho,\o)}
\biggr)d\rho
\end{multline*}
is a smooth solution of \eqref{2-1} for $0\leq \tau<\tau_0$. 
\end{proof}

From \cite[Chapter 6]{10}, one has $F_0(q,\o)\in C^\infty({\mathbb
  R}\times\mathbb S^1)$, $\text{supp}F_0\subseteq (-\infty,M]\times
  \mathbb S^1$, and
\begin{equation}\label{2-7}
|\p_q^k\p_\o^\al F_0(q,\o)|\leq C_{k,\al}(1+|q|)^{-1/2-k}.
\end{equation}
In addition, from the explicit expression for $V(q,\o,\tau)$ we
conclude that, for $\tau\leq b<\tau_0$,
\begin{equation}\label{2-8}
|\p_q^{m+1}\p_\o^\al\p_\tau^l V(q,\o,\tau)|\leq C^b_{m,\al, l}(1+|q|)^{-3/2-m}
\end{equation}
and
\[
|\p_\o^\al\p_\tau^lV(q,\o,\tau)|\leq C^b_{\al,l}.
\]

We now start to construct an approximate solution of \eqref{1-1} for
$0\leq\tau=\ve\sqrt{1+t}<\tau_0$. Let $w_0$ be the solution of the
linear wave equation
\[
\left\{ \enspace
\begin{aligned}
& \p_t^2 w-\Delta w=0,\\
& w(0,x)=\vp_0(x),\quad \p_t w(0,x)=\vp_1(x).
\end{aligned}
\right.
\]
Choose a $C^\infty$ function $\chi(s)$ such that $\chi(s)=1$ for
$s\leq 1$ and $\chi(s)=0$ for $s\geq 2$. We then set, for $0\leq
\tau=\ve\sqrt{1+t}<\tau_0$,
\begin{equation*}
u_a(t,x)=\ve\chi(\ve t)w_0(t,x)+\f{\ve}{\sqrt{r}}
\left(1-\chi(\ve t)\right)\chi(-3\ve q)V(q,\o,\tau).
\end{equation*}

By \cite[Theorem 6.2.1]{10} and Lemma \ref{lem2-3}, one has that
$|Z^\al u_a|\leq C_{\al,b}\ve (1+t)^{-1/2}$ for $\tau\leq b<\tau_0$
and all multi-index $\alpha$. We further set
\[
J_a=\p_t^2 u_a-\Delta
u_a+\sum_{i,j=0}^2(d_{ij}u_a+\sum_{k=0}^2e_{ij}^k\p_ku_a)\p_{ij}u_a
+O(|u_a|^2+|\nabla u_a|^2)\sum_{i,j=0}^2\p_{ij}u_a.
\]
Then one has:

\begin{lemma}\label{lem2-4}
It holds
\[
\int_0^{b^2/\ve^2-1}\|Z^\alpha
J_a\|_{L^2}\,dt\leq C_{\al, b}\ve^{3/2}.
\]
\end{lemma}

\begin{proof}
We divide the proof into three cases.

\textit{Case A.} \  $\ds\f{2}{\ve}\leq t\leq \f{b^2}{\ve^2}-1$.

In this case, $\chi(\ve t)=0$ and
$u_a(t,x)=\ds\f{\ve}{\sqrt{r}}\,\chi(-3\ve q)V(q,\o,\tau)$. Then
\begin{multline*}
J_a=-\f{\ve^2}{r}\biggl(\p_{q\tau}^2\bar{V}
-\Bigl(\sum_{i,j=0}^2d_{ij}\hat{\o}_i\hat{\o}_j\Bigr)\bar{V}\p_q^2\bar{V}
-\Bigl(\sum_{i,j,k=0}^2e_{ij}^k\hat{\o}_i\hat{\o}_j\hat{\o}_k\Bigr)
\p_q\bar{V}\p_q^2\bar{V}
\biggr)\\ +O\bigl(\f{\ve}{(1+t)^{5/2}}\bigr),
\end{multline*}
where $\bar{V}(q,\o,\tau)=\chi(-3\ve q)V(q,\o,\tau)$. In view of
\eqref{2-7}--\eqref{2-8} and the explicit expression for $V$,
\begin{multline*}
\p_{q\tau}\bar{V}
-\Bigl(\sum_{i,j=0}^2d_{ij}\hat{\o}_i\hat{\o}_j\Bigr)\bar{V}\p_q^2\bar{V}
-\Bigl(\sum_{i,j,k=0}^2e_{ij}^k\hat{\o}_i\hat{\o}_j\hat{\o}_k\Bigr)
\p_q\bar{V}\p_q^2\bar{V}\\ =\chi(1-\chi)\p_{q\tau}V-3\ve\chi'\p_\tau
V-\Bigl(\sum_{i,j=0}^2d_{ij}\hat{\o}_i\hat{\o}_j\Bigr)\chi
V\left(9\ve^2\chi''V-6\ve\chi'\p_qV\right)\\
-\Bigl(\sum_{i,j,k=0}^2e_{ij}^k\hat{\o}_i
\hat{\o}_j\hat{\o}_k\Bigr)\Bigl((-3\ve\chi'V+\chi\p_qV)
(9\ve^2\chi''V-6\ve\chi'\p_q V)-3\ve\chi'V\chi\p_q^2V\Bigr),
\end{multline*}
which yields the estimate
\[
  |Z^\al J_a|\leq
  C_{\al,b}\ve(1+t)^{-5/2}+C_{\al,b}\ve^3(1+t)^{-1}\psi(-3\ve q),
\]
where $\psi(s)$ is a cutoff function satisfying $\psi(s)=1$ for $1\leq
s\leq 2$ and $\psi(s)=0$ otherwise.

\smallskip

\textit{Case B.} \ $t\leq \ds\f{1}{\ve}$.

In this case, $\chi(\ve t)=1$ and $u_a=\ve w_0$. This yields
\[
J_a=\ve^2\sum_{i,j=0}^2(d_{ij}w_0+\sum_{k=0}^2e_{ij}^k\p_kw_0)\p_{ij}w_0
+\ve^3O\left(|w_0|^2+|\nabla w_0|^2\right)\sum_{i,j=0}^2\p_{ij}w_0.
\]
It then follows from a direct computation that
\[
|Z^\alpha J_a|\leq C_\alpha\ve^2(1+t)^{-1}(1+|q|)^{-2}.
\]

\smallskip

\textit{Case C.} \ $\ds\f{1}{\varepsilon}\leq t\leq\ds\f{2}{\varepsilon}$.

A direct computation gives
\[
u_a=\ve w_0+\ve \left(1-\chi(\ve t)\right)\left(r^{-1/2}\chi(-3\ve q)V-w_0\right)
\]
and then
\[
J_a=J_1+J_2+J_3+J_4,
\]
where
\begin{align*}
J_1&=\sum_{i,j=0}^2(d_{ij}u_a+\sum_{k=0}^2e_{ij}^k\p_ku_a)\p_{ij}u_a
+O\left(|u_a|^2+|\nabla u_a|^2\right)\sum_{i,j=0}^2\p_{ij}u_a,\\
J_2&=\ve(\p_t^2-\Delta)\Bigl((1-\chi(\ve t))r^{-1/2}\chi(-3\ve q)(V-F_0)
  \Bigr),\\
J_3&=\ve(\p_t^2-\Delta)\Bigl((1-\chi(\ve t))\chi(-3\ve q)(r^{-1/2}F_0-w_0)
  \Bigr),\\
J_4&=\ve(\p_t^2-\Delta)\Bigl((1-\chi(\ve t))(\chi(-3\ve q)-1)w_0\Bigr).
\end{align*}
It is easy to see that
\[
|Z^\alpha J_1|\leq C_{\alpha,b}\ve^2(1+t)^{-1}(1+|q|)^{-1}.
\]
Due to $(\p_t^2-\p_r^2)(V(q,\tau)-F_0(q))=-\p_{\tau
  q}V\ds\f{\ve}{\sqrt{1+t}}+\p_{\tau}^2V\ds\f{\ve^2}{4(1+t)}
-\p_{\tau}V\f{\ve}{4(1+t)^{3/2}}$,
$V(q,0)=F_0(q)$, and the fact that
$(\p_t^2-\Delta)v=r^{-1/2}\bigl((\p_t+\p_r)(\p_t-\p_r)
-r^{-2}(1/4+\p_\o^2)\bigr)(r^{1/2}v)$, one also has that
\[
  |Z^\alpha J_2|\leq C_{\alpha,b}\ve^3(1+|q|)^{-1}.
\]

By the \cite[Theorem 6.2.1]{10}, one has that, for any constant
$l>0$, if $r\geq lt$, then
\begin{equation*}
|Z^\alpha(w_0-r^{-1/2}F_0)|\leq C(1+t)^{-3/2}(1+|q|)^{1/2}.
\end{equation*}
On the other hand, from the fact that
$\p_t^2-\Delta=\ds\f{1}{r+t}\left(S+\o_1H_1+\o_2H_2\right)
(\p_t-\p_r)-\f{1}{r}\p_r-\f{1}{r^2}\Delta\o$, one concludes that
\[
|Z^\alpha J_3|\leq C_{\alpha,b}\ve^2(1+t)^{-3/2}.
\]

Since the support of $J_4$ with respect to the variable $q$ belongs to
the interval $\left(-\infty, -1/(3\ve)\right)$ and applying the fact
that, for any $\phi(t,r)\in C^1$,
\begin{equation}\label{2-11}
|\p \phi|\le\ds\f{C}{1+|t-r|}\,\sum_{|\beta|=1}|Z^{\beta}\phi|,
\end{equation}
one obtains the estimate
\[
  |Z^\alpha J_4|\leq C_{\alpha,b}\ve^{3}(1+t)^{-1/2}.
\]

Collecting the estimates above, one has
\[
|Z^\al J_a|\leq C_{\al,b}\ve^2(1+t)^{-3/2}+C_{\al,b}\ve^2(1+t)^{-1}(1+|q|)^{-1}.
\]
One arrives at
\begin{align*}
&\|Z^\al J_a(\cdot,t)\|_{L^2}\leq C_{\al, b}\ve^{5/2}(1+t)^{-1/2}
+C_{\al, b}\ve(1+t)^{-3/2},\quad \f{2}{\ve}\leq t\leq e^{b/\ve}-1,\\
&\|Z^\al J_a(\cdot,t)\|_{L^2}\leq C_\al\ve^2(1+t)^{-1/2},\quad t\leq\f{2}{\ve}.
\end{align*}

Consequently,
\[
\int_0^{e^{b/\ve}-1}\|Z^\al u_a(\cdot,t)\|_{L^2}\,dt\leq C_{\al,
b}\ve^{3/2},
\]
and Lemma \ref{lem2-4} is proved.

\end{proof}

\smallskip

For latter reference, we quote from \cite{18}:

\begin{lemma}\label{lem2-5}
For $f(t,x)\in C^1({\mathbb R}_+\times\mathbb R^2)$ with
$\operatorname{supp} f\subseteq\{(t, x)\colon r\leq M+t\}$, one has
\[
\|(1+|t-r|)^{-1}f\|_{L^2}\leq C\left\|\p_r f\right\|_{L^2},
\]
where the constant $C>0$ only depends on $M$.
\end{lemma}

Based on these preparations, we next establish:

\begin{proposition}\label{prop2-6}
For $\varepsilon>0$ sufficiently small and $0<\tau=\ve\sqrt{1+t}\leq
b<\tau_0$, \textup{Eq.~\eqref{1-1}} has a $C^\infty$ solution $u(t,x)$
which satisfies, for all $|\alpha|\le 2$,
\begin{equation}\label{2-12}
|Z^\alpha\p (u-u_a)|\leq
C_{b}\varepsilon^{3/2}(1+t)^{-1/2}(1+|t-r|)^{-1/2}.
\end{equation}
\end{proposition}

\begin{proof}
Let $v=u-u_a$. Then
\begin{equation}\label{2-13}
\left\{ \
\begin{aligned}
&\p_t^2 v-\Delta
v+\ds\sum_{i,j=0}^2(d_{ij}u+\ds\sum_{k=0}^2e_{ij}^k\p_ku)\p_{ij}v
+O(|u|^2+|\nabla u|^2)\ds\sum_{i,j=0}^2\p_{ij}v\\
& \quad
=-J_a-\ds\sum_{i,j=0}^2(d_{ij}v +\sum_{k=0}^2e_{ij}^k\p_kv)\p_{ij}u_a\\
&\qquad -O(|v|^2+2|u_av|+|\nabla v|^2+2|\nabla u_a\cdot\nabla
v|)\ds\sum_{i,j=0}^2\p_{ij} u_a,\\
&v(0,x)=\p_t v(0,x)=0.
\end{aligned}
\right.
\end{equation}
We make the induction hypothesis that, for a certain
$T\leq b^2/\ve^2-1$,
\begin{equation}\label{2-14}
|Z^\alpha\p v|\leq\varepsilon(1+t)^{-1/2}(1+|t-r|)^{-1/2},\quad
|\alpha|\leq 2, \ t\leq T,
\end{equation}
which implies that, for $|\alpha|\leq 2$ and $t\leq T$,
\begin{equation}\label{2-15}
|Z^\alpha v|\leq C\varepsilon\left(1+t\right)^{-1/2}(1+|t-r|)^{1/2}.
\end{equation}

To prove the validity of \eqref{2-14}, we will show that, for
$\varepsilon>0$ sufficiently small,
\begin{equation}\label{2-16}
|Z^\alpha\p
v|\leq\f{\ve}{2}\,(1+t)^{-1/2}(1+|t-r|)^{-1/2},\quad
|\alpha|\leq 2, \ t\leq T,
\end{equation}
and then utilize the continuity method to obtain $\ve\sqrt{1+T}=b$.

Applying $Z^\alpha$ on both hand sides of \eqref{2-13} and using
$[Z^{\al}, \p_t^2-\Delta] =\ds\sum_{|\beta|<|\al|}C_{\al\beta}
Z^{\beta}$ $(\p_t^2-\Delta)$ yields, for $|\alpha|\leq 4$,
\begin{equation}\label{2-17}
  \left(\p_t^2-\Delta\right)Z^\alpha v=Z^\alpha
  G-\sum_{|\beta|<|\alpha|}C_{\alpha\beta}Z^\beta G,
\end{equation}
where
\begin{multline*}
G=-\sum_{i,j=0}^2\bigl(d_{ij}u+\sum_{k=0}^2e_{ij}^k\p_ku\bigr)\p_{ij}v
-O\left(|u|^2+|\nabla u|^2\right)\sum_{i,j=0}^2\p_{ij}v-J_a\\
-\sum_{i,j=0}^2(d_{ij}v
+\sum_{k=0}^2e_{ij}^k\p_kv)\p_{ij}u_a -O\left(|v|^2+2|u_av|+|\nabla
v|^2+2|\nabla u_a\cdot\nabla v|\right)\sum_{i,j=0}^2\p_{ij}
u_a.
\end{multline*}
Thus one obtains from \eqref{2-17} that
\begin{multline}\label{2-18}
\p_t^2Z^\al v-\Delta Z^\al v \\
+\sum_{i,j=0}^2\bigl(d_{ij}u+\sum_{k=0}^2e_{ij}^k
\p_ku\bigr)\p_{ij}Z^\al v
+O\left(|u|^2+|\nabla u|^2\right)\sum_{i,j=0}^2\p_{ij}Z^\al v=F,
\end{multline}
where
\begin{multline*}
 F=-\sum_{i,j=0}^2\sum_{\substack{\al_1+\al_2=\al\\ |\al_1|\geq 1}}
Z^{\al_1}\bigl(d_{ij}u+\sum_{k=0}^2e_{ij}^k\p_ku\bigr)Z^{\al_2}\p_{ij} v \\
\begin{aligned}
  & +Z^\al\biggl(-J_a -\sum_{i,j=0}^2\bigl(d_{ij}v
  +\sum_{k=0}^2e_{ij}^k\p_kv\bigr)\p_{ij}u_a \\
  & -O \left(|v|^2+2|u_av|+|\nabla
  v|^2+2|\nabla u_a\cdot\nabla v|\right)\sum_{i,j=0}^2\p_{ij} u_a\biggr) \\
  & -\sum_{i,j=0}^2\bigl(d_{ij}u+\sum_{k=0}^2e_{ij}^k\p_ku\bigr)[Z^\al,\p_{ij}]v
  -O\left(|u|^2+|\nabla
  u|^2\right)\sum_{i,j=0}^2[Z^\al,\p_{ij}]v \\
  & -\sum_{\substack{\al_1+\al_2=\al\\
      |\al_1|\geq 1}}Z^{\al_1}\Bigl(O\left(|u|^2+|\nabla
  u|^2\right)\sum_{i,j=0}^2Z^{\al_2}\p_{ij}v\Bigr)
  -\sum_{|\beta|<|\alpha|}C_{\alpha\beta}Z^\beta G.
\end{aligned}
\end{multline*}

Next we derive an estimate on $\|\p Z^\alpha v\|_{L^2}$ from
Eq.~\eqref{2-18}. Define the energy
\begin{multline*}
E(t)= \f{1}{2}\sum_{|\al|\leq 4}\int_{{\mathbb R}^2}
\biggl(|\p_t Z^\alpha v|^2+|\nabla_x Z^\al v|^2
-\sum_{i,j=0}^2\bigl(d_{ij}u
+\sum_{k=0}^2e_{ij}^k\p_ku\bigr)(\p_iZ^\al v)(\p_jZ^\al v) \\
-O\left(|u|^2+|\nabla u|^2\right)\sum_{k=0}^2(\p_iZ^\al v)(\p_jZ^\al
v)\biggr)\,dx.
\end{multline*}
Multiplying both sides of \eqref{2-18} by $\p_t Z^\alpha v$ ($|\al|\le
4$), integrating by parts, and noting that $|\p^\beta u|=|\p^\beta
u_a+\p^\beta v|\leq C_b\ve(1+t)^{-1/2}$ ($|\beta|=1,2$) from the
construction of $u_{a}$ and assumption \eqref{2-14}, one arrives at
\begin{equation}\label{2-19}
E'(t)\leq \f{C_b\ve}{\sqrt{1+t}}\,E(t)
+\sum_{|\al|\leq 4}\int_{{\mathbb R}^2}|F||\p_t Z^\al v|\,dx.
\end{equation}

We now treat each of the terms appearing in the integral
$\ds\sum_{|\al|\leq 4}\int_{{\mathbb R}^2}|F||\p_t Z^\al v|\,dx$
separately.

\smallskip

\textit{(A) \ Terms
  $\ds\sum_{\substack{\al_1+\al_2=\beta\\|\al_1|\geq 1}} \int_{\mathbb
    R^2}\Bigl|Z^{\al_1}\bigl(d_{ij}u+\sum_{k=0}^2e_{ij}^k\p_ku+O\left(|u|^2+|\nabla
  u|^2\right)\bigr)Z^{\al_2}\p_{ij} v\Bigr| |\p_t Z^\alpha v|\,dx$ with
  $|\beta|\leq|\alpha|$.}

It suffices to estimate $\ds\sum_{\substack{\al_1+\al_2=\beta\\|\al_1|\geq 1}} \int_{\mathbb
R^2}\Bigl|Z^{\al_1}\bigl(d_{ij}u+\sum_{k=0}^2e_{ij}^k\p_ku\bigr)Z^{\al_2}\p_{ij} v
\Bigr||\p_t Z^\alpha v|\,dx$. Note that:

(i) \ By assumption \eqref{2-15}, one has, for $|\delta|\le 2$,
\begin{equation}\label{2-20}
\bigl|(1+|t-r|)^{-1}Z^{\delta}(u_a+v)\bigr|\le \ds\f{C_b\,\ve}{\sqrt{1+t}}.
\end{equation}

(ii) \ By \eqref{2-20} and \eqref{2-11}, one has, for
$|\al_1|+|\al_2|=|\beta|\leq 4$ with $|\al_1|\ge 1$,
\begin{multline}\label{2-21}
\int_{{\mathbb R}^2}\bigl|Z^{\al_1}(d_{ij}u+e_{ij}^k\p_ku)(Z^{\al_2}\p_{ij} v)
\p_t Z^\alpha v\bigr|\,dx\\
\begin{aligned}
&\leq C_b\int_{\mathbb R^2}(|Z^{\al_1}u_a|+|Z^{\al_1}\p_ku_a|)\cdot|(Z^{\al_2}\p_{ij} v)
(\p_t Z^\al v)|\,dx \\
& \qquad
+C\int_{{\mathbb R}^2}(|Z^{\al_1}v|+|Z^{\al_1}\p_kv|)\cdot|(Z^{\al_2}
\p_{ij} v)(\p_t Z^\al v)|\,dx\\
& \leq\f{C_{b}\,\ve}{\sqrt{1+t}}\,E(t)+C\sum_{|\g|\leq
|\al_2|+1}\int_{\mathbb R^2}|(1+|t-r|)^{-1}Z^{\al_1} v|\cdot|Z^\g\p v|\cdot|\p_t Z^\al
v|\,dx\\
&\qquad+C\sum_{|\g|\leq|\al_2|+1}\int_{\mathbb
R^2}|Z^{\al_1} \p_kv|\cdot|Z^\g\p v|\cdot|\p_t Z^\al
v|\,dx.
\end{aligned}
\end{multline}
Note that there is at most one number larger than 2 between $|\al_1|$
and $|\g|$. If $|\al_1|>2$, then $|\g|\le 2$. Thus, by
Lemma~\ref{lem2-5} applied to $(1+|t-r|)^{-1}Z^{\al_1} v$ and
assumption \eqref{2-14}, one arrives at
\begin{multline}\label{2-22}
\int_{{\mathbb R}^2}|(1+|t-r|)^{-1}Z^{\al_1} v|\cdot|Z^\g\p v|\cdot|\p_t Z^\al v|\,dx\\
+\int_{\mathbb R^2}|Z^{\al_1} \p_kv|\cdot|Z^\g\p v|\cdot|\p_t Z^\al
v|\,dx
\leq \f{C_b\,\ve}{\sqrt{1+t}}\,E(t).
\end{multline}
If $|\g|>2$, then $|\al_1|\le 2$. It follows from \eqref{2-15} that
$\bigl|(1+|t-r|)^{-1}Z^{\al_1} v\bigr|\leq C\ve(1+t)^{-1/2}(1+|t-r|)^{-1/2}$
 which leads to
\begin{multline}\label{2-23}
\int_{{\mathbb R}^2}|(1+|t-r|)^{-1}Z^{\al_1} v|\cdot|Z^\g\p v|\cdot|\p_t Z^\al v|\,dx \\
+\int_{{\mathbb R}^2}|Z^{\al_1} \p_kv|\cdot|Z^\g\p v|\cdot|\p_t Z^\al v|\,dx
\leq \f{C_b\,\ve}{\sqrt{1+t}}\,E(t).
\end{multline}
Substituting \eqref{2-22}--\eqref{2-23} into \eqref{2-21} yields
\begin{equation}\label{2-24}
\sum_{\substack{\al_1+\al_2=\beta\\|\al_1|\geq 1}} \int_{\mathbb
R^2}\Bigl|Z^{\al_1}\bigl(d_{ij}u+\sum_{k=0}^2e_{ij}^k\p_ku \bigr)Z^{\al_2}
\p_{ij} v\Bigr||\p_t Z^\alpha v|\,dx \le \f{C_b\,\ve}{\sqrt{1+t}}\,E(t).
\end{equation}

\smallskip

\textit{(B) \ Terms $\ds\int_{\mathbb
    R^2}\Bigl|Z^\beta\bigl((d_{ij}u+\ds\sum_{k=0}^2e_{ij}^k\p_ku+O(|u|^2+|\nabla
  u|^2))\p_{ij} v\bigr)\cdot\p_tZ^\al v\Bigr|\,dx$ with
  $|\beta|<|\alpha|$.}

We only need to treat the term $\int_{\mathbb R^2}\Bigl|\bigl(
d_{ij}u+\ds\sum_{k=0}^2e_{ij}^k\p_ku\bigr)Z^\beta\p_{ij}
v\cdot\p_tZ^\al v\Bigr|\,dx$, since the other terms have been estimated in
(A). By \eqref{2-14}, one has
\begin{multline}\label{2-25}
\int_{{\mathbb
    R}^2}\Bigl|\bigl(d_{ij}u+\sum_{k=0}^2e_{ij}^k\p_ku\bigr)
Z^\beta\p_{ij} v\cdot\p_tZ^\al v\Bigr|\,dx \\
\leq C\sum_{|\g|\leq|\beta|+1\leq|\al|}\int_{{\mathbb
    R}^2}|(1+|t-r|)^{-1}u| \cdot|Z^\g\p v|\cdot|\p_tZ^\al v|\,dx\\
+C\int_{{\mathbb R}^2}|\p u|\cdot|Z^\beta\p^2
v|\cdot|\p_tZ^\al v|\,dx \leq\f{C_{b}\,\ve}{\sqrt{1+t}}\,E(t).
\end{multline}

\smallskip

\textit{(C) \ Terms $\ds\int_{{\mathbb R}^2}|\p_t Z^\al
  v|\cdot|Z^\beta J_a|\, dx$ with $|\beta|\leq|\al|\leq 4$.}

In this case, one has
\begin{equation}\label{2-26}
\int_{{\mathbb R}^2}|\p_t Z^\al v|\cdot|Z^\beta J_a|\,dx
\leq\|Z^\beta J_a(\cdot,t)\|_{L^2}\sqrt{E(t)}.
\end{equation}

\smallskip

\textit{(D) \ Terms $\ds\int_{{\mathbb
      R}^2}\Bigl|Z^\beta\bigl(\bigl(d_{ij}v
  +\ds\sum_{k=0}^2e_{ij}^k\p_kv+O\left(|v|^2+2|u_av|+|\nabla
  v|^2+2|\nabla u_a\cdot\nabla v|\right)\bigr)$ $\p_{ij}
  u_a\bigr)\Bigr|\cdot|\p_t Z^\al v|\, dx$ with $|\beta|\leq|\al|\leq
  4$.}

A direct computation yields
\begin{multline}\label{2-27}
\int_{{\mathbb R}^2}\Bigl|Z^\beta\bigl(\bigl(d_{ij}v +\sum_{k=0}^2e_{ij}^k\p_kv
 +O\left(|v|^2+2|u_av|+|\nabla v|^2
+2|\nabla u_a\cdot\nabla v|\right)\bigr)\p_{ij} u_a\bigr)\Bigr| \\
\cdot |\p_t Z^\al v|\,dx
\leq C\sum_{\substack{|\beta_1|+|\beta_2|\leq|\beta|+1\\|\beta_1|\leq|\beta
|}}\int_{{\mathbb R}^2}\big(|(1+|t-r|)^{-1}(Z^{\beta_1}v)(Z^{\beta_2}\p u_a)| \\
+|(Z^{\beta_1}\p v)(Z^{\beta_2}\p u_a)|\big)|\p_t Z^\al v|\,dx
\leq\f{C_{b}\,\ve}{\sqrt{1+t}}\,E(t).
\end{multline}

\smallskip

\textit{(E) \ Terms $\ds\int_{{\mathbb R}^2}\Bigl|\bigl(d_{ij}v
+\ds\sum_{k=0}^2e_{ij}^k\p_kv+O\left(|u|^2+|\nabla
u|^2\right)\bigr)[Z^\al,\p_{ij}]v\Bigr|\cdot|\p_t Z^\al v|\,dx$.}

Since $[Z^\al,\p_{ij}]=\ds\sum_{|\beta|\leq|\al|-1}
C_{\al\beta}^{ij}\p^2Z^\beta$, one has
\begin{multline}\label{2-28}
\int_{{\mathbb R}^2}\Bigl|\bigl(d_{ij}u
+\sum_{k=0}^2e_{ij}^k\p_ku+O\left(|u|^2+|\nabla u|^2\right)\bigr)
[Z^\al,\p_{ij}]v\Bigr|\cdot|\p_t Z^\al v|\,dx\\
\begin{aligned}
& \leq C\sum_{|\beta|\leq|\al|-1}\int_{{\mathbb R}^2}(|u|+|\p u|)|\p^2Z^\beta v|
\cdot|\p_t Z^\al v|\,dx\\
& \leq C\sum_{|\g|\leq|\al|}\int_{{\mathbb R}^2}|(1+|t-r|)^{-1}
(|u|+|\p u|)|\cdot|\p Z^\g v|\cdot|\p_tZ^\al v|\,dx\\
& \leq\f{C_{b}\,\ve}{\sqrt{1+t}}\,E(t).
\end{aligned}
\end{multline}

Substituting \eqref{2-24}--\eqref{2-28} into \eqref{2-19} yields
\begin{equation*}
E'(t)\leq\f{C_{b}\,\ve}{\sqrt{1+t}}\,E(t)
+\sum_{|\beta|\leq 4}\|Z^\beta J_a(\cdot,t)\|_{L^2}\sqrt{E(t)}.
\end{equation*}

By Lemma~\ref{lem2-4} and Gronwall's inequality, one obtains
\[
\|\p Z^\al v\|_{L^2}\leq C_{b}\,\ve^{3/2},\quad|\al|\leq 4,
\]
and then
\begin{equation}\label{2-30}
\|Z^\al\p v\|_{L^2}\leq C_{b}\,\ve^{3/2},\quad|\al|\leq 4.
\end{equation}
By \eqref{2-30} and the Klainerman-Sobolev inequality (see \cite{10,
  16}), one has
\begin{equation}\label{2-31}
|Z^\al\p v|\leq
C_{b}\,\ve^{3/2}(1+t)^{-1/2}(1+|t-r|)^{-1/2},\quad
|\alpha|\leq 2,\ t\leq T,
\end{equation}
which means that, for small $\ve>0$,
\[
|Z^\alpha\p
v|\leq\f{\ve}{2}\,(1+t)^{-1/2}(1+|t-r|)^{-1/2},\quad
|\alpha|\leq 2, \ t\leq T.
\]
In this way, we have completed the proofs of \eqref{2-16} and further
of \eqref{2-12} along with \eqref{2-31}.  
\end{proof}

Proposition~\ref{prop2-6} immediately gives that
 $\ds\varliminf_{\ve\rightarrow 0}\ve(1+
T_\ve)^{1/2}\geq\tau _0$, hence
\begin{equation}\label{2-32}
\ds\varliminf_{\ve\rightarrow 0}\ve\sqrt{ T_\ve}\geq\tau _0.
\end{equation}


\section{Proof of the Geometric Blowup Theorem}

In this section, we use the coordinates $(r, \th, t)$ instead of $(x,
t)$ to study Eq.~\eqref{1-1}.

We set
\[
\si=r-t, \quad \tau=\ve\sqrt{t},
\]
and write $u(t,x)=\ds\f{\ve}{\sqrt r}\,G(\si,\th,\tau)$
for $r>0$. Further we introduce the notation
\begin{align*}
&\widetilde\o=(0,-\sin\th,\cos\th),\quad\bar\o=(0,\cos\th,\sin\th), \\
&\hat\o=(-1,\cos\th,\sin\th),\quad R=\tau^2+\ve^2\si,\\
&a_{ij}=d_{ij}G+\sum_{k=0}^2e_{ij}^k\biggl(\hat{\o}_k\p_\si
G+\ve^2\Bigl(-\f{\bar{\o}_k}
{2R}G+\f{1}{R}\widetilde{\o}_k\p_\th G+\f{1}{2\tau}\delta_0^k\p_\tau G
\Bigr)\biggr)\\
&\qquad
+O\biggl(\f{\ve^2}{R^{1/2}}G^2+\sum_{k=0}^2\f{\ve^2}{R^{1/2}}(-\f{\ve^2}
{2R}\bar{\o}_kG+\hat{\o}_k\p_\si G+\f{\ve^2}{R}\widetilde{\o}_k\p_\th G
+\f{\ve^2}{2\tau}\delta_0^k\p_\tau G)^2\biggr).
\end{align*}
It then follows from a direct computation that Eq.~\eqref{1-1} takes
the form
\begin{equation}\label{3-1}
P(G)\equiv\sum_{i,j=0}^2p_{ij}(y,G,\nabla_y
G)\p_{ij}G+\ve^2p_0(y,G,\nabla_y G)=0
\end{equation}
where
\begin{align*}
&p_{00}=\sum_{i,j=0}^2a_{ij}\hat\o_i\hat\o_j,\quad
p_{01}=p_{10}=\f{\ve^2}{2R}\sum_{i,j=0}^2a_{ij}(\widetilde\o_i\hat\o_j
+\widetilde\o_j\hat\o_i),\\
&p_{02}=p_{20}=\f{\ve^2}{4\tau}\sum_{i,j=0}^2a_{ij}
(\hat\o_j\delta_0^i+\hat\o_i\delta_0^j)-\f{R^{1/2}}{2\tau},\quad
p_{11}=-\f{\ve^2}{R^{3/2}}+\f{\ve^4}{R^2}\sum_{i,j=0}^2a_{ij}\widetilde\o_i
\widetilde\o_j,\\
&p_{12}=p_{21}=\f{\ve^4}{4R\tau}\sum_{i,j=0}^2a_{ij}
(\widetilde\o_j\delta_0^i+\widetilde\o_i\delta_0^j),\quad
p_{22}=\f{\ve^2R^{1/2}}{4\tau^2}+\f{\ve^4}{4\tau^2}\sum_{i,j=0}^2a_{ij}
\delta_0^i\delta_0^j
\end{align*}
with $y=(y_0,y_1,y_2)=(\si,\th,\tau)$, $\p_i=\p_{y_i}$ $(0\le i\le 2)$,
and $p_0$ is a smooth function.

Introduce a transformation $\Phi$,
\begin{equation}\label{3-2}
\Phi(s,\th,\tau)= \left(\si,\th,\tau\right).
\end{equation}
where $\si=\phi(s,\th,\tau)$. Set
$w(s,\th,\tau)=G(\phi(s,\th,\tau),\th,\tau)$ and $v(s,\th,\tau)=\p_\si
G(\phi(s,\th,\tau),\th,\tau)$. It is obvious that
$\p_sw=v\p_s\phi$. Now, if we find smooth functions $(\phi, w, v)$ and
some point $m_{\ve}=(s_{\ve}, \th_{\ve}, \tau_{\ve})$ such that
$\p_s\phi(m_{\ve})=0$ and $\p_sv(m_{\ve})\neq 0$, then the second
order derivatives of $G$ has a singularity at $M_{\ve}=(\phi(s_{\ve},
\th_{\ve}, \tau_{\ve}), \th_{\ve}, \tau_{\ve})$,
$\p_\si^2G=\ds\f{\p_sv}{\p_s\phi}$. In \cite{1, 2}, such a blowup is
said to be of geometric type.

The following proposition is established in \cite{2} by direct
computation. It guides us to the construction of the blowup system of
\eqref{1-1}.

\begin{proposition}\label{prop3-1}
With $\bar{\p}=(0,\p_1,\p_2)$ and
$\hat\phi=(\hat\phi_0,\hat\phi_1,\hat\phi_2) =(-1,\p_1\phi,\p_2\phi)$,
one has
\begin{align*}
&(\nabla_y G)(\Phi)=\bar{\p}w-\hat\phi v,\\
&(\p_{ij}G)(\Phi)=\bar{\p}_{ij}w-v\bar{\p}_{ij}\phi
-\left(\hat\phi_i\bar{\p}_jv+\hat\phi_j\bar{\p}_iv\right)
+\hat\phi_i\hat\phi_j\bigl(\f{\p_sv}{\p_s\phi}\bigr),\\
&P(G)(\Phi)=\f{\p_sv}{\p_s\phi} \, I_1+I_2,
\end{align*}
where
\begin{align*}
I_1&=\sum p_{ij}(\phi,\th,\tau,w,\bar{\p}w-\hat\phi v)\hat\phi_i\hat\phi_j,\\
I_2&=\sum p_{ij}(\phi,\th,\tau,w,\bar{\p}w-\hat\phi v)
\bigl(\bar{\p}_{ij} w-v\bar{\p}_{ij}\phi
-(\hat\phi_i\bar{\p}_jv+\hat\phi_j\bar{\p}_iv)\bigr) \\
& \qquad +\ve^2p_0(\phi,\th,\tau,w,\bar{\p}w-\hat\phi v).
\end{align*}
\end{proposition}

Using this proposition, one sees that, in order to solve the nonlinear
equation $P(G)=0$, it suffices to solve the system
\begin{equation}\label{3-3}
\left\{ \enspace
\begin{aligned}
& I_1=0,\\
& I_2=0,\\
& I_3=\p_sw-v\p_sv=0
\end{aligned}
\right.
\end{equation}
 for $(\phi, w, v)$. This system is called the blowup system of
 \eqref{1-1} in the terminology used in \cite{1, 2}.

We now comment on the existence of local solution of \eqref{3-3}.

From the analysis of Section~2, we know that $P(G)=0$ can be solved
with the corresponding initial data on $t=\tau_1^2/\ve^2$ (since
\eqref{1-1} has a unique smooth solution for $\tau<\tau_0$ and small
$\ve$) in the strip
\[
D_S=\{(\si,\th,\tau)\colon -C_0\le \si\le M,\,
\th_0-\dl_0\le\th\le\th_0+\dl_0,\, \tau_1\le\tau\le \tau_1+\eta\},
\]
where $C_0>0$ is some large constant, $\tau_1>0$, and $\eta>0$
is so small that $\eta<\tau_0-\tau_1$.

From $I_1=\ds\sum p_{ij}(\phi,\th,\tau,G(\phi,\th,\tau),\nabla
G(\phi,\th,\tau))\hat\phi_i\hat\phi_j=0$, one has that $\ds\f{\p
  I_1}{\p(\p_\tau\phi)}=\f{R^{1/2}}{\tau}+O(\ve^2)>0$ for $\ve>0$
small and a smooth function $\phi$. By the implicit function theorem,
one then obtains
\begin{equation}\label{3-4}
\p_\tau\phi=E(\ve,\th,\tau,\phi,\p_{\th}\phi),
\end{equation}
where $E$ is a smooth function of its arguments.

With initial data $\phi(s,\th,\tau_1)=s$, \eqref{3-4} has a unique
solution $\bar{\phi}$ for $\eta>0$ sufficient small. Set
\[
\bar{w}=G(\bar\phi,\th,\tau),\quad\bar{v}=\p_\si G(\bar\phi,\th,\tau).
\]
One has that $(\bar\phi,\bar{w},\bar{v})$ is a local solution of the
blowup system \eqref{3-3}, as the local existence of $G$ is known by
\eqref{2-32}. Moreover, from the uniqueness result for the solution
$u(t,x)$ of \eqref{1-1} for $t\in [0, (\tau_1+\eta)^2/\ve^2]$, one has
that $\bar{v}$ and $\overline{\phi}-s$ are smooth and flat on
$\{s=M\}$.

In order to solve the blowup system \eqref{3-3}, as in \cite{1,2}, we
shall use the Nash-Moser-H\"ormander iteration technique under
hypothesis \eqref{H}. This construction is split into five steps.


\subsection{Structure of the linearization of the blowup system}

Let $(\dot\phi, \dot w, \dot v)$ be the unknown solution of the
linearization of the blowup system \eqref{3-3}. As in
\cite[Thm.~3]{2}, set $\dot z=\dot w- v\dot\phi$. Recall that
$\ds\sum_{i,j,k=0}^2e_{ij}^k\p_ku\p_{ij}u$ does not satisfy the null
condition. It then follows by direct computation that the
linearization of system \eqref{3-3} can be changed into the system
\begin{align}
\mathcal L_1(\dot\phi, \dot z)&\equiv Z_1\p_s\dot z-\ve^2(\p_s\phi) Q\dot
z+a_0 \p_s\dot z+\ve^2a_1\p\dot z+a_2\dot z+b_1Z_1\dot\phi+b_2\dot\phi
=\dot f_1,\label{3-5}\\
\mathcal L_2(\dot\phi, \dot z)&\equiv
Z_1^2\dot\phi+a_3Z_1\dot\phi+a_4\dot\phi+\ve^2c_0Q\dot
z \notag \\
&\qquad +\ve^2Z_1(a_5\p_\tau+a_6\p_\th)\dot z+\ve^2a_7\p\dot z+a_8Z_1\dot
z+\ve^2a_9\dot z=\dot f_2,\label{3-6}
\end{align}
where
\begin{align*}
Z_1&=\sum p_{ij}(\hat\phi_i\bar\p_j+\hat\phi_j\bar\p_i)=(1+O(\ve))
(\p_\tau+O(\ve^2)\p_\th),\\
Q&=\ve^{-2}\sum p_{ij}\bar\p_{ij}=(-1+O(\ve^2))\p_\th^2+O(\ve^2)\p_{\th\tau}
+\bigl(\f{1}{4\tau}+O(\ve)\bigr)\p_\tau^2,\\
b_1&=\p_sv+\bigl(\sum_{i,j,k}e_{ij}^k\hat\o_i\hat\o_j\hat\o_k\bigr)^{-1}
(Z_1\p_s\phi+\sum_{i,j}d_{ij}\hat\o_i\hat\o_jv\p_s\phi)+O(\ve^{1/2}),\\
b_2&=Z_1\p_sv+b_1\sum_{i,j}d_{ij}\hat\o_i\hat\o_jv+O(\ve^{1/2}),\\
c_0&=\sum_{i,j,k}e_{ij}^k\hat\o_i\hat\o_j\hat\o_k+O(\ve^2),
\end{align*}
and the $a_i$ ($0\le i\le 9$) are smooth functions.

On the other hand, $\dot v$ is determined from the first equation
$I_1'(\dot\phi, \dot w, \dot v)=\dot f_3$ in the linearization of the
blowup system \eqref{3-3}, see \cite[Proposition~II.2]{2}.

To obtain a weighted energy estimate for \eqref{3-5}-\eqref{3-6}, we
choose a ``nearly horizontal" surface $\Sigma$ passing through
$\{\tau=\tau_1,\,s=M\}$, as in \cite{2}, in place of the initial plane
$\{\tau=\tau_1\}$, where $\Sigma$ is a characteristic surface for the
operator $Z_1\p_s-\ve^2\p_s{\bar\phi}Q$ with its coefficients computed
using $(\bar{\phi},\bar{v},\bar{w})$. Note that if the characteristic
surface $\Sigma$ is defined by the equation $\tau=\psi(s,\th)+\tau_1$,
then $\psi$ fulfills
\begin{equation}\label{3-7}
\left\{ \enspace
\begin{aligned}
& \biggl(-1+O(\ve^2)\p_{\th}\psi\biggr)\p_s\psi \\
& \qquad -\ve^2\p_s{\bar\phi}\biggl(\ds\f{1}
{4\tau}+O(\ve)-O(\ve^2)\p_\th\psi+(-1+O(\ve^2))(\p_{\th}\psi)^2\biggr)=0,\\
& \psi(M,\th)=0.
\end{aligned}
\right.
\end{equation}
It is readily seen that, for $\ve>0$ small, Eq.~\eqref{3-7} has a
smooth solution $\psi(s,\th)$ in the domain $D_S$.

Choose a cut-off function $\chi\in C^\infty({\mathbb R})$ with
$\chi(t)=1$ for $t\leq 1/2$ and $\chi(t)=0$ for $t\geq 1$ and perform
the change of variables
\begin{equation}\label{3-8}
X=s,\quad Y=\th,\quad T=\tau-\tau_1-\psi(s,\th)
\chi\biggl(\f{\tau-\tau_1}{\eta}\biggr).
\end{equation}
We will then work in the domain
\[
D_1=\left\{(X, Y, T)\colon -C_0\le X\le M, \, \th_0-\dl_0\le
Y\le\th_0+\dl_0, \, 0\le T\le \tau_\ve-\tau_1\right\},
\]
which is still unknown, because we do not know yet the precise value
of $\tau_{\ve}$. By \eqref{3-8}, the characteristic surface $\Sigma$
becomes $\{T=0\}$.


\subsection{Construction of an approximate solution of (3.3)}

In a first step of the Nash-Moser-H\"or\-mander iteration method, one is
required to construct an approximate solution $(\phi_a, w_a, v_a)$ of
\eqref{3-3} such that, at some point, $\phi_a$ satisfies \eqref{H}.

For $\ve=0$, the blowup system \eqref{3-3} becomes
\begin{equation}\label{3-9}
\left\{ \enspace
\begin{aligned}
& \p_T\phi+\ds\sum_{i,j}\biggl(d_{ij}w+\sum_ke_{ij}^k\hat\o_kv\biggr)
\hat\o_i\hat\o_j=0,\\
& -\p_Tv=0,\\
& \p_Xw-v\p_X\phi=0
\end{aligned}
\right.
\end{equation}
with the initial and boundary conditions
\[
\begin{aligned}
&\phi(X,Y,0)=X,\quad
\phi(M,Y,T)=M,\\
&v(X,Y,0)=\p_\si F_0(\bar s(X,Y,\tau_1),Y),\quad v(M,Y,T)=0,\\
\end{aligned}
\]
and
$\bar s(X,Y,\tau_1)$ being given by
\[
X=M+\ds\int_M^{\bar
s}\biggl\{\exp\bigl(-F_1(\rho,Y)\tau_1\bigr)
\biggl(1+\ds\f{F_2(\rho,Y)}{F_1(\rho,Y)}\biggr)
-\f{F_2(\rho,Y)}{F_1(\rho,Y)}\biggr\}\,d\rho.
\]
An exact solution of \eqref{3-9} is
\[
\left\{ \enspace
\begin{aligned}
& \bar\phi_0=M+\ds\int_M^{\bar
s(X,Y,\tau_1)}\biggl\{\exp\bigl(-(T+\tau_1)F_1(\rho,Y)\bigr)
\biggl(1+\ds\f{F_2(\rho,Y)}{F_1(\rho,Y)}\biggr)
-\f{F_2(\rho,Y)}{F_1(\rho,Y)}\biggr\}\,d\rho,\\
& \bar v_0=\p_\si F_0(\bar s(X,Y,\tau_1),Y),\\
& \bar w_0=\ds\int_M^{\bar s(X,Y,\tau_1)}\p_\si
F_0(\rho,Y)\biggl\{\exp\bigl(-(T+\tau_1)F_1(\rho,Y)\bigr)\\
& \qquad\qquad
\biggl(1+\ds\f{F_2(\rho,Y)}{F_1(\rho,Y)}\biggr)
-\f{F_2(\rho,Y)}{F_1(\rho,Y)}\biggr\}\,d\rho.
\end{aligned}
\right.
\]

Note that \eqref{3-3} has a local solution $(\bar{\phi},\bar w,\bar
v)$ for $0\le T\le \eta$ whose existence is guaranteed by previous
statements. We now glue $(\bar{\phi},\bar w,\bar v)$ to
$(\bar{\phi}_0,\bar w_0,\bar v_0)$ to obtain an approximate solution
of \eqref{3-3}
\begin{equation}\label{3-10}
\left\{ \enspace
\begin{aligned}
&\phi_a(X,Y,T)=\chi(T/\eta)\bar{\phi}(X,Y,T)
+\bigl(1-\chi(T/\eta)\bigr)\bar{\phi}_0(X,Y,T),\\
&v_a(X,Y,T)=\chi(T/\eta)\bar{v}(X,Y,T)
+\bigl(1-\chi(T/\eta)\bigr)\bar{v}_0(X,Y,T),\\
&w_a(X,Y,T)=\chi(T/\eta)\bar{w}(X,Y,T)
+\bigl(1-\chi(T/\eta)\bigr)\bar{w}_0(X,Y,T).
\end{aligned}
\right.
\end{equation}

Substituting the approximate solution $(\phi_a, w_a, v_a)$ into the
blowup system \eqref{3-3} yields
\[
  I_i=f_a^i, \quad 1\leq i\leq 3,
\]
where $f_a^i$ is smooth, flat on $\{X=M\}$ and zero near $\{T=0\}$. In
addition, under assumption (ND), one can show that $\phi_a$
satisfies \eqref{H} at the point $(\bar\si,\th_0,\tau_0-\tau_1)$, where
\[
  \bar\si=M+\ds\int_M^{\si_0}\biggl(\exp\bigl(-F_1(\rho,\th_0)\tau_1\bigr)
  \Bigl(1+\f{F_2(\rho,\th_0)}{F_1(\rho,\th_0)}\Bigr)
 -\f{F_2(\rho,\th_0)}{F_1(\rho,\th_0)}\biggr)\,d\rho.
\]
This will be our next step.


\subsection{Condition (H) for the approximate solution \boldmath $\phi_a$}

We now prove that the approximate solution $\phi_a$ given by
\eqref{3-10} satisfies condition \eqref{H} at the point
$(\bar\si,\th_0,\tau_0-\tau_1)$.

It follows from a direct computation that
\[
\p_X\bar\phi_0(X,Y,T)=\Biggl(\ds\f{\ds\exp(-(T+\tau_1)F_1)
\biggl(1+\f{F_2}{F_1}\biggr)-\f{F_2}{F_1}}
{\ds\exp(-\tau_1F_1)\biggl(1+\f{F_2}{F_1}\biggr)-\f{F_2}{F_1}}\Biggr)
(\bar s(X,\tau_1,Y),Y).
\]
Recall that $(\si_0,\th_0)$ is the interior minimum point of
$G_0(\si,\th)$. Thus one has that $\ds\nabla_{\si,\th}G_0(\si_0,\th_0)
=0$. Then $\nabla_{s,\th}\p_sq(\th_0,\tau_0,\si_0)=\ds
F_2(\si_0,\th_0)\nabla_{\si,\th}G_0(\si_0,\th_0)=0$ which follows from
the expression for $q(\th,\tau,s)$ in Lemma~\ref{lem2-3}. With
$\p_sq(\th_0,\tau_0,\si_0)=0$, this yields
\[
\nabla_{X,Y}(\p_X\bar\phi_0)(\bar\si,\th_0,\tau_0-\tau_1)=0.
\]
A direct calculation shows that
$\na_{s,\th}^2\p_sq(\th_0,\tau_0,\si_0)=\ds
F_2(\si_0,\th_0)\na_{\si,\th}^2 G_0(\si_0,\th_0)$ is a positive
definite matrix. On the other hand, at the point
$M_0=(\bar\si,\th_0,\tau_0-\tau_1)$, one has
\begin{align*}
&\p_X^2\p_X\bar\phi_0(M_0) =\f{\p_s^2\p_sq(\th_0,\tau_0,\si_0)
(\bar s_X)^2(\si_0,\tau_1,\th_0)}{\Bigl(\exp\bigl(-\tau_1F_1\bigr)
\bigl(1+\f{F_2}{F_1}\bigr)-
\f{F_2}{F_1}\Bigr)(\si_0,\th_0)},\\
&\p_{XY}\p_X\bar\phi_0(M_0)\\
&\hspace{10mm}=\f{\p_s^2\p_sq(\th_0,\tau_0,\si_0)
(\bar s_X\bar s_Y)(\si_0,\tau_1,\th_0)+\p_{s\th}^2\p_sq(\th_0,\tau_0,\si_0)
\bar s_X(\si_0,\tau_1,\th_0)}{\Bigl(\exp\bigl(-\tau_1F_1\bigr)
\bigl(1+\f{F_2}{F_1}\bigr)-
\f{F_2}{F_1}\Bigr)(\si_0,\th_0)},\\
&\p_Y^2\p_X\bar\phi_0(M_0)\\
&\hspace{10mm}=\f{\p_s^2\p_sq(\th_0,\tau_0,\si_0)
(\bar s_Y)^2(\si_0,\tau_1,\th_0)+2\p_{s\th}\p_sq(\th_0,\tau_0,\si_0)
\bar s_Y(\si_0,\tau_1,\th_0)}{\Bigl(\exp\bigl(-\tau_1F_1\bigr)
\bigl(1+\f{F_2}{F_1}\bigr)-
\f{F_2}{F_1}\Bigr)(\si_0,\th_0)}\\
&\qquad\qquad +\f{\p_\th^2\p_sq(\th_0,\tau_0,\si_0)}{\Bigl(\exp\bigl(-\tau_1F_1\bigr)
\bigl(1+\f{F_2}{F_1}\bigr)-
\f{F_2}{F_1}\Bigr)(\si_0,\th_0)}
\end{align*}
Therefore, $\na_{X,Y}^2\p_X\bar\phi_0(M_0)$ is a positive definite matrix.

\smallskip

Next we verify condition \eqref{H} for the function $\phi_a$.

\smallskip

(i) \ Because of $\p_X\bar{\phi}(X,Y,0)=1$, one can assume that
$\p_X\bar{\phi}(X,Y,T)>0$ for $T\leq\eta$. In addition, it follows
from the expression for $\p_X\bar{\phi}_0(X,Y,T)$ that
$\p_X\bar\phi_0(X,Y,T)\geq 0$ and $\p_X\bar\phi_0(X,Y,T)=0$ if and
only if $(X,Y,T)=(\bar\si,\th_0,\tau_0-\tau_1)$.  Therefore,
$\p_X\phi_a(X,Y,T)\geq 0$ holds. Moreover, $\p_X\phi_a(X,Y,T)=0$ if
and only if $T\geq\eta$ and $\p_X\bar{\phi}_0(X,Y,T)=0$. This means
that
\[
\p_X\phi_a(X,Y,T)=0 \enspace \Longleftrightarrow \enspace (X,Y,T)
=(\bar\si,\th_0,\tau_0-\tau_1).
\]

\smallskip

(ii) \ Since $\eta<\tau_0-\tau_1$, $\phi_a(X,Y,T)=
\bar\phi_0(X,Y,T)$ holds in a neighborhood of
$(\bar\si,\th_0,\tau_0-\tau_1)$. Thus,
\begin{align*}
&\p_T\p_X\phi_a(\bar\si,\th_0,\tau_0-\tau_1)=
\p_T\p_X\bar\phi_0(\bar\si,\th_0,\tau_0-\tau_1)<0,\\
&\na_{X,Y}^2\phi_a(\bar\si,\th_0,\tau_0-\tau_1)=
\na_{X,Y}^2\bar\phi_0(\bar\si,\th_0,\tau_0-\tau_1)\enspace
\text{is positive definite}.
\end{align*}


\subsection{Reduction to a Goursat problem in a fixed domain}

To be free to adjust the height of the domain $D_1$, we perform a
change of variables depending on a small nonnegative parameter $\la$,
\begin{equation}\label{3-11}
X=x,\quad Y=y,\quad T=T(\rho,\la)=(\tau_0-\tau_1)
(\rho+\la\rho(1-\chi_1(\rho))),
\end{equation}
where $\chi_1(\rho)$ equals $1$ for $\rho$ near $0$ and $0$ for $\rho$
near $1$. We will then work in a fixed subdomain $D_2$ of $D_1$,
\[
D_2=\{(x,y,\rho)\colon -C_0\le x\le M, \, \th_0-\dl_0\le y\le\th_0+\dl_0, \,
0\le\rho\le 1\}.
\]
For $\la=\la_0=0$, the approximate solution of \eqref{3-3} is
\begin{equation*}
\left\{ \enspace
\begin{aligned}
\phi_0(x,y,\rho)&=\phi_a(x,y,(\tau_0-\tau_1)\rho),\\
v_0(x,y,\rho)&=v_a(x,y,(\tau_0-\tau_1)\rho),\\
w_0(x,y,\rho)&=w_a(x,y,(\tau_0-\tau_1)\rho).
\end{aligned}
\right.
\end{equation*}
Moreover, $\phi_0$ satisfies condition \eqref{H} in $D_2$ at the point
$(\bar\si, \th_0, 1)$.

On the characteristic surfaces $\{x=M\}$ and $\{\rho=0\}$,
respectively, of Eq.~\eqref{3-3}, we impose the boundary
conditions
\begin{equation*}
\text{$\phi$ is flat on $\{x=M\}$  and $\phi-\phi_0$ is flat on
$\{\rho=0\}$.}
\end{equation*}

We now turn our attention to the linearized equations \eqref{3-5} and
\eqref{3-6}. Under the changes of variables \eqref{3-8} and
\eqref{3-11}, it follows from a direct computation that \eqref{3-5}
and \eqref{3-6} assume the form
\begin{align}
&ZS\dot z-\ve^2(S\phi) N\dot z+\al_1 S\dot z+\ve^2l_1(\na \dot z)
+\al_2\dot z+\beta_1Z\dot\phi+\beta_2\dot\phi
=\dot F_1,\label{3-14}\\
&Z^2\dot\phi+\al_3Z\dot\phi+\al_4\dot\phi+\ve^2\g_0N\dot z \notag \\
& \qquad +\ve^2Z(\al_5Z+\al_6\p_y)\dot z
+\ve^2l_2(\na\dot z)+\al_7Z\dot z+\al_8\dot z=\dot F_2,\label{3-15}
\end{align}
where
\begin{align*}
&Z=\p_\rho+\ve^2z_0\p_y,\quad S=\p_s=\p_x+\ve^2\p_\rho,\quad
N=N_1Z^2+2\ve^2N_2Z\p_y+N_3\p_y^2,\\
&\beta_1=b_1, \quad \beta_2=(\p_\rho
T)b_2+O(\ve),\quad\g_0=(\p_\rho T)c_0+O(\ve),
\end{align*}
$N_1 =\ds\f{1}{4\tau\p_\rho T}+O(\ve)>0$ and $N_3=-\p_\rho
T+O(\ve)<0$; $l_1(\na\dot z)$ and $l_2(\na\dot z)$ are linear
combinations of $\nabla\dot z$.

We specifically point out that although \eqref{3-14} and \eqref{3-15}
are somehow similar to the linearized equations $(3.1.1_a)$ and
$(3.1.1_b)$ of \cite{2}, the coefficients $\al_1$ and $\al_2$ in
\eqref{3-14} are only bounded quantities different from the ones in
$(3.1.1_a)$ of \cite{2} which are of order $O(\ve^2)$. In addition,
there are more terms in \eqref{3-15} than in $(3.1.1_b)$ of \cite{2}
due to the simultaneous appearance of the solution $u$ and its
first-order derivatives $\na u$ in the coefficients of \eqref{1-1}. In
view of the differences between \eqref{3-14}--\eqref{3-15} and
$(3.1.1_a)-(3.1.1_b)$ of \cite{2}, we will derive energy estimates on
the solutions of \eqref{3-14}--\eqref{3-15} directly by choosing
suitable multipliers and then integrating by parts different from
changing the main part of \eqref{3-14} into a third-order scalar
equation to derive related estimates by introducing a new unknown
function $\dot k$ with $\dot z=Z\dot k$ as in \cite{1,2}.

In the process of solving \eqref{3-14}--\eqref{3-15} we are required
to choose a subdomain $D_3$ of $D_2$ which is an domain of influence
for the first-order differential operator $Z$, contains the point
$(\bar\si, \th_0,1)$, and is bounded by the planes $\{x=-C_0\}$,
$\{x=M\}$, $\{y=0\}$, $\{\rho=0\}$, $\{\rho=1\}$, $S_+$, and
$S_-$. Here, $S_+$ and $S_-$ do not intersect in $D_2$, and their
normal directions are $(-\eta,\nu,1)$ and $(-\eta,-\nu,1)$,
respectively, $\nu>0$ is an appropriate constant. In addition, it is
assumed that we are given a smooth function $\phi$ on $D_3$ and a
constant $\la$ in \eqref{3-11} close to $\phi_0$ and $\la_0$,
respectively, where the function $\phi$ also satisfies \eqref{H} at
some point $(\bar x_0, \bar y_0, 1)$. (This is achieved invoking the
implicit function theorem established in \cite{1} in terms of the
properties of $\phi_0$ satisfying \eqref{H} at the point
$(\bar\si,\th_0,1)$.)


\subsection{The tame estimate and solvability of (3.12)--(3.13)}

Note that $\rho=0$ is a characteristic surface of the operator
$ZS-\ve^2(S\phi)N$. Then
\begin{equation}\label{3-16}
s_0-(S\phi)N_1=0.
\end{equation}

As in \cite{1}, set
\[
A=S\phi,\quad\delta=1-t,\quad g=\exp h(x-t),\quad
p=\delta\sqrt g,\quad |\cdot|_0=\|\cdot\|_{L^2(D_3)}.
\]
Then one obtains the following energy estimate:
\begin{lemma}\label{lem3-2}
There exist $C>0$, $\ve_0>0$, $\eta_0>0$ and $h_0>0$ such that for
$\phi$ satisfying \textup{\eqref{H}}, \eqref{3-16}, and
\begin{equation}\label{3-17}
|\phi-\phi_0|_{C^4(D_3)}
+|w-w_0|_{C^4(D_3)}+|v-v_0|_{C^4(D_3)}\leq\eta_0,
\end{equation}
one has, for all $0\leq\ve\leq\ve_0$ and $h\geq h_0$,
\begin{multline}\label{3-18}
\int_{D_3}\f{1}{A}\delta g(1+\delta h)(S\dot z)^2+\ve^2\int_{D_3}
\delta g(A+\delta h)(\p_y\dot z)^2+h|pZ\dot z|_0^2+|pZ\dot\phi|_0^2
+h|p\dot\phi|_0^2\\
\leq C|p\dot F_1|_0^2+C\ds\int_{D_3}\delta^4g|\dot F_2|^2,
\end{multline}
where the functions $\dot\phi$ and $\dot z$ are smooth and flat on
$\{t=0\}$ and $\{x=M\}$.
\end{lemma}

\begin{proof}
Set $P=ZS-\ve^2AN$ and choose the multiplier $\mathcal M\dot z=aS\dot
z+dZ\dot z$ as in \cite{1}, where the functions $a$ and $d$ are to be
determined later. By an integration by parts, one obtains
\begin{multline}\label{3-19}
\int_{D_3}(P\dot z)(\mathcal M\dot z)\,dxdyd\rho
=\int_{D_3} K_1(S\dot z)^2+\int_{D_3} K_2(\p_y\dot z)^2+\int_{D_3}
K_3(Z\dot z)^2\\ +\int_{D_3} K_4(S\dot z)(\p_y\dot z)
+\int_{D_3} K_5(S\dot z)(Z\dot z)+\int_{D_3} K_6(\p_y\dot z)(Z\dot z) \\
+I_1-I_0-J_1+L_++L_-,
\end{multline}
where
\begin{align*}
K_1&=-\f12(Za)-\f12\ve^2(\p_yz_0)a, \\
K_2&=-\f12\ve^2S(AaN_3)-\f12\ve^2Z(AdN_3)-\ve^8AaN_2(\p_\rho s_0)z_0
-\ve^6AaN_2(\p_xz_0)\\
&\qquad-\ve^{10}z_0^2(\p_ys_0)AaN_2-\ve^8s_0(\p_\rho z_0)AaN_2
-\f12\ve^4(\p_\rho s_0)AaN_3\\
&\qquad-\ve^6AaN_3(\p_y s_0)z_0
-\f12\ve^4(\p_yz_0)AdN_3+\ve^4AdN_3(\p_y z_0),\\
K_3&=-\ds\f12(Sd)-\f12\ve^2(\p_\rho s_0)d+\ve^2d(\p_\rho s_0)
+\ve^4z_0(\p_ys_0)d-\f12\ve^4(\p_\rho s_0)AaN_1\\
&\qquad-\f12\ve^2S(AaN_1)
+\ve^4AaN_1(\p_\rho s_0)+\ve^6AaN_2(\p_ys_0)+\f12\ve^4(\p_yz_0)AdN_1\\
&\qquad+\f12\ve^2Z(AdN_1)+\ve^4\p_y(AdN_2),\\
K_4&=\ve^6(\p_yz_0)AaN_2+\ve^4Z(AaN_2)+\ve^6AaN_2(\p_yz_0)+\ve^2\p_y(AaN_3),\\
K_5&=\ve^4AaN_1(\p_yz_0)+\ve^2Z(AaN_1)+\ve^4\p_y(AaN_2),\\
K_6&=-\ve^4d(\p_\rho s_0)z_0-\ve^2d(\p_xz_0)-\ve^6z_0^2(\p_ys_0)d
-\ve^4ds_0(\p_\rho z_0)-\ve^6AaN_1(\p_\rho s_0)z_0,\\
&\qquad-\ve^6(\p_\rho s_0)AaN_2-\ve^4S(AaN_2)-\ve^8AaN_2(\p_ys_0)z_0
+\ve^6AaN_2(\p_\rho s_0)\\
&\qquad+\ve^8z_0(\p_ys_0)AaN_2
+\ve^4AaN_3(\p_ys_0)+2\ve^6AdN_2(\p_yz_0)+\ve^2\p_y(AdN_3),\\
\intertext{and}
I_1&=\int_{\{t=1\}}\biggl\{\f12a(S\dot
z)^2+\f12\ve^2s_0d(Z\dot z)^2-\ve^2AaN_1(Z\dot z)(S\dot
z)+\f12\ve^4s_0AaN_1(Z\dot z)^2\\
&\qquad-\ve^4AaN_2(\p_y\dot z)(S\dot z),
+\ve^6s_0AaN_2(\p_y\dot z)(Z\dot
z)+\f12\ve^4s_0AaN_3(\p_y\dot z)^2\\
&\qquad-\f12\ve^2AdN_1(Z\dot z)^2
+\f12\ve^2AdN_3(\p_y\dot z)^2\biggr\}\,dS,\\
I_0&=0,\\
J_1&=\int_{\{x=-C_0\}}\biggl\{\f12d(Z\dot
z)^2+\f12\ve^2AaN_1(Z\dot z)^2
+\ve^4AaN_2(\p_y\dot z)(Z\dot
z)
\\&\qquad
+\f12\ve^2AaN_3(\p_y\dot z)^2\biggr\}\,dS,\\
L_\pm&=\ds\int_{S_\pm}\biggl\{\f12a(S\dot z)^2\mp\f12\ve^2\nu
az_0(S\dot z)^2-\f12\eta d(Z\dot z)^2+\f12\ve^2s_0d(Z\dot
z)^2\\
&\qquad-\ve^2AaN_1(Z\dot z)(S\dot z)\pm\ve^4\nu AaN_1z_0(Z\dot z)(S\dot z)
-\f12\ve^2\eta AaN_1(Z\dot z)^2\\
&\qquad+\f12\ve^4s_0AaN_1(Z\dot z)^2-\ve^4AaN_2(\p_y\dot z)(S\dot z)
\pm\ve^6\nu z_0AaN_2(\p_y\dot z)(S\dot z)\\
&\qquad-\ve^4\eta AaN_2(\p_y\dot z)(Z\dot z)
+\ve^6s_0AaN_2(\p_y\dot z)(Z\dot z)\pm\ve^4\nu AaN_2(S\dot z)(Z\dot z)\\
&\qquad\pm\ve^2\nu AaN_3(\p_y\dot z)(S\dot z)-\f12\ve^2\eta AaN_3(\p_y\dot z)^2
+\f12\ve^4s_0AaN_3(\p_y\dot z)^2\\
&\qquad-\f12\ve^2AdN_1(Z\dot z)^2\pm\f12\ve^4\nu z_0AdN_1(Z\dot z)^2
\pm\ve^4\nu AdN_2(Z\dot z)^2\\
&\qquad\pm\ve^2\nu AdN_3(\p_y\dot z)(Z\dot z)+\f12\ve^2AdN_3(\p_y\dot z)^2
\mp\f12\ve^4\nu AdN_3z_0(\p_y\dot z)^2\biggr\}\,dS.
\end{align*}
Choosing $a=A^{-1}\delta^2g$ and $d=-\delta^2g$ and employing
condition \eqref{H} on $\phi$, \eqref{3-16}, and geometric properties of
$D_3$, as in the proof of \cite[Proposition~3.3]{2} one obtains by
\eqref{3-19} and a careful computation that
\[
I_1=\ds\int_{\{t=1\}}\f12a(S\dot z)^2dS\geq 0,\quad J_1\leq 0,\quad L_\pm\geq 0
\]
and
\begin{equation}\label{3-20}
\int_{D_3}\f{1}{A}\delta g(1+\delta h)(S\dot z)^2+\ve^2\int_{D_3}
\delta g(A+\delta h)(\p_y\dot z)^2+h\left|pZ\dot z\right|_0^2\leq
C\left|pP\dot z\right|_0^2.
\end{equation}
Note that, for any $\vartheta>0$, there exists a $C>0$ such that, for
all $h\geq 1$ and all smooth $\psi$ satisfying $\psi|_{t=0}=0$, the
inequality
\begin{equation}\label{3-21}
\int_{D_3}(\delta^{\vartheta-1}+h\delta^\vartheta)g\psi^2\leq
C\int_{D_3}\delta^{\vartheta+1}g(Z\psi)^2
\end{equation}
holds. In addition,
\begin{equation}\label{3-22}
P\dot z=\dot F_1-\al_1 S\dot z-\ve^2l_1(\na \dot z)
-\al_2\dot z-\beta_1Z\dot\phi-\beta_2\dot\phi.
\end{equation}
Substituting \eqref{3-22} into the right hand side of \eqref{3-20} and
using \eqref{3-21} for $\vartheta=2$ and the function $\dot z$ instead
of $\psi$, one obtains, for sufficiently large $h>0$,
\begin{multline}\label{3-23}
\int_{D_3}\f{1}{A}\delta g(1+\delta h)(S\dot z)^2+\ve^2\int_{D_3}
\delta g(A+\delta h)(\p_y\dot z)^2+h\left|pZ\dot z\right|_0^2 \\
\leq C\left|p(\dot F_1-\beta_1Z\dot\phi-\beta_2\dot\phi)\right|_0^2,
\end{multline}
where we note that the ``largeness'' of $\al_1$ and $\al_2$ does not
play a role, as the parameter $h>0$ can be chosen as large as
required.

Next we estimate $\beta_1$ and $\beta_2$ in \eqref{3-14} and
\eqref{3-23}. Upon substituting $(\bar{\phi}_0,\bar w_0,\bar v_0)$
into the expressions for $\beta_1$ and $\beta_2$, a direct computation
yields
\[
\beta_1(\bar{\phi}_0,\bar w_0,\bar v_0)=O(\ve^{1/2}),\quad
\beta_2(\bar{\phi}_0,\bar w_0,\bar v_0)=O(\ve^{1/2}).
\]

On the other hand, by estimate \eqref{2-12} and in the coordinate system
$(X,Y,T)$ of \eqref{3-8}, one has
\begin{equation}\label{3-24}
\left\{ \enspace
\begin{aligned}
&\p_T\bar\phi+\ds\sum_{i,j}\biggl(d_{ij}\bar w+
\ds\sum_ke_{ij}^k\hat\o_k\bar v\biggr)\hat\o_i\hat\o_j+O(\ve)=0,\\
&-\p_T\bar v+O(\ve)=0,\\
&\p_X\bar w-\bar v\p_X\bar\phi=0,\\
&\bar\phi(X,Y,0)=\bar\phi_0(X,Y,0),\\
&\bar v(X,Y,0)=\bar v_0(X,Y,0)+O(\ve^{1/2}),\\
&\bar\phi(M,Y,T)=\bar\phi_0(M,Y,T)=M,\\
&\bar v(M,Y,T)=\bar v_0(M,Y,T)=0.
\end{aligned}
\right.
\end{equation}
Then it follows from the expressions for $({\phi}_0, w_0, v_0)$,
$\beta_1$, $\beta_2$, \eqref{3-24}, and a direct computation that
\begin{equation*}
\beta_1({\phi}_0,w_0,v_0)=O(\ve^{1/2}),\quad
\beta_2({\phi}_0,w_0,v_0)=O(\ve^{1/2}).
\end{equation*}
This together with \eqref{3-17} yields that $\beta_1(\phi,w,v)$ and
$\beta_2(\phi,w,v)$ are small when $\eta_0>0$ is small.

By \eqref{3-14}, one has
\begin{equation*}
\ve^2 N\dot z=\ds\f{1}{A}\bigl(ZS\dot z+\al_1 S\dot z+\ve^2l_1(\na
\dot z)+\al_2\dot z+\beta_1Z\dot\phi +\beta_2\dot\phi-\dot F_1).
\end{equation*}
Substituting this
into \eqref{3-15} and utilizing \eqref{3-21}
for $\vartheta=3$ yields
\begin{multline*}
\left|p(Z\dot\phi+\al_3\dot\phi+\ve^2\al_5Z\dot z+\ve^2\al_6\p_y\dot z
+\al_7\dot z+\g_0A^{-1}S\dot z)\right|_0^2\\
\leq C\ds\int\delta^4g\bigl\{|\dot F_2|^2+A^{-2}|\dot F_1|^2
+A^{-2}|\dot\phi|^2+A^{-4}|S\dot z|^2 \\
+\ve^4A^{-2}|\nabla\dot z|^2+A^{-2}|\dot z|^2
+\beta_1^2A^{-2}|Z\dot\phi|^2\bigr\}.
\end{multline*}
Noting $\delta\leq CA$, it follows that
\begin{multline*}
\left|pZ\dot\phi\right|_0^2\leq C\int\delta^4g|\dot
F_2|^2\\
+C\int\delta^2g\bigl\{|\dot F_1|^2+|\dot\phi|^2 +A^{-2}|S\dot
z|^2+\ve^4|\nabla\dot z|^2+|\dot
z|^2+\beta_1^2|Z\dot\phi|^2\bigr\}.
\end{multline*}
Making use of the smallness of $\beta_1$ and the inequality
$h\left|p\dot\phi\right|_0^2\leq C\left|Z\dot\phi\right|_0^2$, one has
then that
\begin{multline*}
\left|pZ\dot\phi\right|_0^2+h|p\dot\phi|_0^2\leq C\int\delta^4g|\dot
F_2|^2\\+C\int\delta^2g\bigl\{|\dot F_1|^2+A^{-2}|S\dot z|^2
+\ve^4|\nabla\dot z|^2\bigr\}+Ch^{-1}\left|pZ\dot z\right|_0^2.
\end{multline*}
Combining this with \eqref{3-23} completes the proof of
\eqref{3-18}. 
\end{proof}

Next we establish higher-order tame estimates on the solutions of
\eqref{3-14}--\eqref{3-15}.

\begin{lemma}\label{lem3-3}
There exist $\ve_0>0$, $\eta_0>0$ and an integer $n_0$ such that for
smooth functions $(\phi,w,v)$ satisfying\/ \textup{\eqref{H}}, \eqref{3-16},
and
\[
|\phi-\phi_0|_{C^4(D_3)}
+|w-w_0|_{C^4(D_3)}+|v-v_0|_{C^4(D_3)}\leq\eta_0,
\]
one has, for all integers $s$ and all $0\leq\ve\leq\ve_0$,
\begin{multline}\label{3-30}
|\dot\phi|_{H^s(D_3)}+|\dot z|_{H^s(D_3)} \\
\leq C_s\bigl(|\phi|_{H^{s+n_0}(D_3)}, |w|_{H^{s+n_0}(D_3)},
|v|_{H^{s+n_0}(D_3)}\bigr)\bigl(|\dot F_1|_{H^{s}(D_3)}
+|\dot F_2|_{H^{s}(D_3)}\bigr),
\end{multline}
where $\dot\phi$ and $\dot z$ are the solutions of
\eqref{3-14}--\eqref{3-15} which are smooth and flat on $\{t=0\}$ and
$\{x=M\}$.
\end{lemma}

\begin{proof}
Let $\mathcal T=\{Z,S,\p_y\}$. For $k\in\mathbb N$ and $T\in\mathcal T$,
one has by a direct computation that
\[
[T^l,P]
=\ve^2\sum_{\substack{a+b+c=l+1\\b\leq l}}
C_{abc}^1Z^aS^b\p_y^c+\ve^2\sum_{a+b+c\leq l}C_{abc}^2Z^aS^b\p_y^c
\]
and
\[
[T^l,Z^2]=\ve^2\sum_{\substack{a+b+c=l+1\\a\geq 1,a+b\geq 2}}D_{abc}^1
Z^aS^b\p_y^c+\ve^2\sum_{a+b+c\leq l}D_{abc}^2Z^aS^b\p_y^c,
\]
where $C^i_{abc}$ and $D^i_{abc}$ ($i=1,2$) are smooth functions of
$(\phi, w, v)$.

By taking the derivatives of up to order $l$ on both sides of
\eqref{3-14}--\eqref{3-15}, one arrives at
\[
\left\{ \enspace
\begin{aligned}
&(P+\al_1S+\ve^2l_1\nabla+\al_2)T^l\dot z+\beta_1 ZT^l\dot\phi
+\beta_2T^l\dot\phi+\ve^2\sum_{\substack{a+b+c=l+1\\b\leq l}}
\bar C_{abc}^1Z^aS^b\p_y^c\dot z\\
&\qquad +\sum_{a+b+c\leq l}\bar C_{abc}^2Z^aS^b\p_y^c\dot z+\sum_{a+b+c\leq l}
\bar C_{abc}^3Z^aS^b\p_y^c\dot\phi=T^l\dot F_1,\\
&Z^2T^l\dot\phi+\al_3ZT^l\dot\phi+\al_4T^l\dot\phi+\ve^2\g_0NT^l\dot z
+\ve^2Z(\al_5Z+\al_6\p_y)T^l\dot z+\ve^2l_2(\na T^l\dot z)\\
&\qquad +\al_7ZT^l\dot z+\al_8T^l\dot z+\ve^2\sum_{a+b+c=l}
\bar D_{abc}^1Z^{a+1}S^b\p_y^c\dot\phi+\sum_{a+b+c\leq l}
\bar D_{abc}^2Z^aS^b\p_y^c\dot\phi\\
&\qquad +\ve^2\sum_{a+b+c=l+1}\bar D_{abc}^3Z^aS^b\p_y^c\dot z+\sum_{a+b+c\leq l}
\bar D_{abc}^4Z^aS^b\p_y^c\dot z=T^l\dot F_2,
\end{aligned}
\right.
\]
where $\bar C^i_{abc}$ ($1\le i\le 3$) and $\bar D^j_{abc}$ ($1\le
j\le 4$) are smooth functions of $(\phi, w, v)$.

Applying Lemma \ref{lem3-2} to $(T^l\dot z, T^l\dot\phi)$ yields by a
direct computation that
\begin{multline*}
\sum_{T\in\mathcal T}\biggl(\quad\ds\int_{D_3}\f{1}{A}\delta
  g(1+\delta h)(ST^l\dot z)^2 \\
+\ve^2\int_{D_3}\delta g(A+\delta
  h)(\p_yT^l\dot z)^2+h|pZT^l\dot
  z|_0^2+|pZT^l\dot\phi|_0^2+h|pT^l\dot\phi|_0^2\biggr)\\
 \leq
  C\sum_{T\in\mathcal T}\bigl(|pT^l\dot
  F_1|_0^2+C\ds\int_{D_3}\delta^4g|T^l\dot F_2|^2\bigr).
\end{multline*}
Note that the space $\tilde H^s=\{f\in L^2(D_3)\colon T^lf\in
L^2(D_3), l\leq s\}$ is the usual Sobolev space $H^s(D_3)$. For
$(\phi, w, v)\in H^{s+n_0}(D_3)$ with a suitably large $n_0\in\mathbb N$,
when utilizing the fact that $|pv|_0$ is equivalent to $|v|_0$, one
gets the desired tame estimate \eqref{3-30}.  
\end{proof}

As in \cite[Proposition~3.4]{2}, based on Lemmas
\ref{lem3-2}--\ref{lem3-3}, one obtains the following result by the
standard Picard iteration and a fixed-point argument:

\begin{lemma}\label{lem3-4}
Let $(\phi, w, v)$ and $\ve$ satisfy the assumptions of\/
\textup{Lemma~\ref{lem3-3}}. Then, for all smooth $\dot F_1$ and $\dot
F_2$ which are flat on $\{t=0\}$ and $\{x=M\}$, there exists a unique
smooth solution $(\dot\phi,\dot z)$ of \eqref{3-14}--\eqref{3-15}
which is flat on $\{t=0\}$ and $\{x=M\}$. Moreover, $(\dot\phi,\dot
z)$ satisfies the tame estimate \eqref{3-30}.
\end{lemma}

Consequently, by Lemmas \ref{lem3-2}--\ref{lem3-3}, the standard
Nash-Moser-H\"ormander iteration technique (see \cite{1, 2}),
and Sobolev's embedding theorem, one completes the proof of
Theorem~\ref{thm1-2} in a certain domain $\mathcal D_0$ (where
$\mathcal D_0$ here is the domain $D_3$). \qed


\section{Proof of the Main Theorem}

Utilizing Theorem~\ref{thm1-2}, we now prove Theorem~\ref{thm1-1}.

\smallskip

(i) \ It holds $u(t,x)\in C^1(\Phi(\mathcal D_3))$ and
$\|u\|_{C^1(\Phi(\mathcal D_3))}\le C\ve^2$.

\smallskip

To this end, we will show $G(\si,\th,\tau),\,
\p_{\si}G(\si,\th,\tau)\in C(\Phi(\mathcal D_3))$.  Without loss of
generality, only $G(\si,\th,\tau) \in C(\Phi(\mathcal D_3))$ is
proved. In fact, it is enough to show that $G$ is continuous at the
point $M_{\ve}=(\si_\ve, \th_\ve, \tau_{\ve})\equiv (\phi(m_\ve),
\th_\ve, \tau_{\ve})$ in view of $\phi,w\in C^3(\mathcal D_3)$ and
\eqref{H} of Theorem~\ref{thm1-2}. Let $(\si_n, \th_n, \tau_n)\in
\Phi(\mathcal D_3)$ satisfy $(\si_n, \th_n, \tau_n)\to
(\si_\ve,\th_\ve,\tau_\ve)$ as $n\to\infty$. It then follows from
\eqref{H} that there is a unique point $(s_n, \th_n, \tau_n)\in
\mathcal D_3$ such that $\si_n=\phi(s_n, \th_n, \tau_n)$. By Taylor's
formula, one has
\begin{multline*}
\si_n-\si_\ve=\na_{\th,\tau} \phi(m_{\ve})\cdot
(\th_n-\th_\ve,
\tau_n-\tau_\ve)+ \p_{s\tau}^2\phi(m_{\ve})(s_n-s_\ve)
(\tau_n-\tau_\ve) \\
\begin{aligned}
&+ \f12\,(\th_n-\th_\ve,
\tau_n-\tau_\ve)\na^2_{\th,\tau}\phi(m_\ve)(\th_n-\th_\ve,
\tau_n-\tau_\ve)^T\\
&+ \f16\,\p_s^3\phi(m_\ve)(s_n-s_\ve)^3
+o(|s_n-s_\ve|^3)+o(|\th_n-\th_\ve| +|\tau_n-\tau_\ve|).
\end{aligned}
\end{multline*}
Together with $\p_s^3\phi(m_\ve)>0$, this yields $s_n\to s_\ve$ as
$n\to\infty$. Therefore, one obtains $G\in C(\Phi(\mathcal D_3))$ from
$G(\Phi)=w$ and the continuity of $\phi, w$ in $\mathcal D_3$. Thus,
it follows that $u(t,x)=\ds\f{\ve}{\sqrt{r}}\,G(r-t,\th, \ve\sqrt
t)\in C^1(\Phi(\mathcal D_3))$ and $\|u\|_{C^1(\Phi(\mathcal D_3))}\le
C\ve^2$.

\smallskip

(ii) \ It holds $\ds \f{1}{C(T_\ve-t)}\leq\|\na_{t,x}^2
u(t,\cdot)\|_{L^\infty(\Phi(D_3))}\leq\f{C}{T_\ve-t}$.

Recall that we have obtained the $C^3$ solutions $(\phi, w, v)$ of
\eqref{3-3} in the domain $\mathcal D_3$ by Theorem
\ref{thm1-2}. Therefore, the solution of \eqref{1-1} is obtained in
the domain $\Phi(\mathcal D_3)$ in the coordinate system
$(s,\th,\tau)$. Now we go back to the original coordinate system $(r,
\th, t)$.

For $(s,\th,\tau)\in \mathcal D_3$ close to the point $m_\ve$, by
Taylor's formula, there exists $\bar\tau=\bar\la\tau+
(1-\bar\la)\tau_\ve$ with $0<\bar\la<1$ such that
\begin{equation}\label{4-1}
\p_s\phi(s,\th,\tau)=\p_s\phi(s,\th,\tau_\ve)
+\p_{\tau s}\phi(s,\th,\bar\tau)(\tau-\tau_\ve).
\end{equation}
Another $\bar\eta\in(0,1)$ makes the point $(\bar
s,\bar\th)=(\bar\eta
s+(1-\bar\eta)s_\ve,\bar\eta\th+(1-\bar\eta)\th_\ve)$ to satisfy
\begin{multline}\label{4-2}
\p_s\phi(s,\th,\tau)=\f12(s-s_\ve,\th-\th_\ve)\nabla_{s,\th}^2\p_s\phi(\bar
s,\bar\th,\tau_\ve)(s-s_\ve,\th-\th_\ve)^T \\
+\p_{\tau s}\phi(s,\th,\bar\tau)(\tau-\tau_\ve).
\end{multline}
In addition, we can assume that $-2c_0\leq\p_{s\tau}^2\phi\leq-c_0$ in
$\mathcal D_3$ because of $\p_{s\tau}\phi(m_{\ve})<0$ and $\phi\in
C^3(\mathcal D_0)$; here $c_0>0$ is a constant. Together with
\eqref{4-1} and \eqref{H}, this yields
\begin{equation*}
\p_s\phi(s,\th,\tau)\geq
c_0(\tau_\ve-\tau)=c_0\ve\,\f{T_\ve-t}{\sqrt{T_\ve}+\sqrt{t}}\geq
\f{c_0\ve}{3}\cdot\f{T_\ve-t}{\sqrt{t}}.
\end{equation*}

On the other hand, using the fact that
$\nabla_{s,\th}^2\p_s\phi(m_\ve)>0$ if $|(s-s_\ve,\th-\th_\ve)|
<\tau_{\ve}-\tau$, from \eqref{4-2} it is readily seen that
\begin{equation*}
|\p_s\phi(s,\th,\tau)|\le 3c_0(\tau_{\ve}-\tau)\leq
\f{3c_0 \ve(T_\ve-t)}{2\sqrt t}.
\end{equation*}

From $u=\ds\f{\ve}{\sqrt r}\,G$ and $v=\p_\si G$, one then has
\begin{equation*}
\f{1}{C(T_\ve-t)}\leq\|\na_{t,x}^2
u(t,\cdot)\|_{L^\infty(\Phi(D_3))}\leq\f{C}{T_\ve-t}.
\end{equation*}

\smallskip

(iii) \ It holds $u(t,x)\in C^1([0, T_{\ve}]\times{\mathbb R}^2)\cap
C^2(([0, T_{\ve}]\times{\mathbb R}^2) \setminus \{M_{\ve}\})$ and
$\|u(t,x)\|_{C^1([0, T_{\ve}]\times{\mathbb R}^2)}\le C\ve$.

\smallskip

For $t\le T_{\ve}$ away from $M_{\ve}$, due to assumption (ND), the
smooth solution of \eqref{2-1} does not blow up in $\left(\{t\le
T_{\ve}\}\times{\mathbb R}^2\right)\setminus\{M_{\ve}\}$. Therefore,
similar to the proof on Proposition~\ref{prop2-6}, one obtains
$u(t,x)\in C^2(([0, T_{\ve}]\times{\mathbb R}^2) \setminus
\{M_{\ve}\})$. Furthermore, in the domain $\left(\{t\le
T_{\ve}\}\times{\mathbb R}^2\right)\setminus\{\Phi(\mathcal D_3)\}$,
one has
\[
|u|\leq C\ve \enspace \text{and} \enspace |\na_{t,x}^\al u|\leq C\ve(1+t)^{-1/2}
\enspace \text{for $|\al|=1,2$.}
\]
Together with (i), this yields $\|u(t,x)\|_{C^1([0,
    T_{\ve}]\times{\mathbb R}^2)}\le C\ve$.

\smallskip

(iv) \ It holds $\ds\lim_{\ve\to 0}\ve \sqrt T_{\ve}=\tau_0$.

\smallskip

By Theorem~\ref{1-2} and the corresponding Nash-Moser-H\"ormander
iteration process, one infers that $\ds\lim_{\ve\to
  0}\tau_{\ve}=\tau_0$ for the solution $u(t,x)$ with variables $(r-t,
\th, \ve\sqrt t)$ in $\Phi(\mathcal D_3)$. This implies that the
lifespan $T_{\ve}$ satisfies
\begin{equation}\label{4-6}
\varlimsup_{\ve\to 0}\ve \sqrt T_{\ve}\le\tau_0.
\end{equation}

\eqref{2-32} and \eqref{4-6} together yield
\[
  \lim_{\ve\to 0}\ve \sqrt T_{\ve}=\tau_0.
\]
Collecting (i)-(iv) completes the proof of Theorem~\ref{thm1-1}. \qed



\end{document}